\documentclass[11pt,a4paper]{amsart}
\usepackage{latexsym}
\usepackage{amssymb,amsmath,amsthm,amscd,amsfonts,enumerate}
\usepackage{graphicx}
\usepackage{color}
\usepackage{pdflscape}
\usepackage{longtable}
\usepackage{epsfig}
\usepackage{setspace}
\usepackage{tikz}
\usepackage{adjustbox}
\usepackage{mathtools}

\usepackage[utf8]{inputenc}
\usepackage[T1]{fontenc}
\usepackage{lmodern}

\usepackage[english]{babel}

\usepackage{times}
\usepackage{cite}
\usepackage{ulem}
\usepackage[mathcal]{euscript}
\usepackage{hyperref}
\usepackage{cancel}
\usepackage{multirow}
\usetikzlibrary{arrows}

\numberwithin{equation}{section}

\theoremstyle{plain}
\newtheorem{theorem} {Theorem} [section]

\newtheorem{lemma} [theorem] {Lemma}

\newtheorem{proposition} [theorem] {Proposition}

\newtheorem{definition}[theorem]{Definition}
\newtheorem{remark}[theorem]{Remark}

\newtheorem*{notation} {Notation}
\newcommand \g {\mathfrak{g}}
\newcommand \ab {\operatorname{ab}}
\newcommand \gl {\mathfrak{gl}}
\newcommand \GL {\operatorname{GL}}
\renewcommand \O {\operatorname{O}}

\newcommand \F {\mathbb{F}}

\newcommand \z {\mathfrak{z}}

\newcommand \h {\mathfrak{h}}

\newcommand \C {\mathbb{C}}

\newcommand \N {\mathbb{N}}
\newcommand \Z {\mathbb{Z}}
\renewcommand \t {\mathfrak{t}}

\newcommand \LS {\mathcal{LS}_{(m|n)}}
\newcommand \NLS {\mathcal{N}_{(m|n)}}
\newcommand \rk {\operatorname{rank}}

\newcommand \diag {\operatorname{diag}}
\newcommand \End {\operatorname{End}}
\newcommand \Aut {\operatorname{Aut}}
\newcommand \Mat {\operatorname{Mat}}
\newcommand \Span {\operatorname{Span}}

\newcommand \Sym {\operatorname{Sym}}
\newcommand \I {\operatorname{I}}
\newcommand \id {\operatorname{id}}
\def\[{[\![}
\def\]{]\!]}

\begin{document}

\title{Varieties of Nilpotent Lie Superalgebras of dimension $\leq 5$}

\author{Mar\'ia Alejandra Alvarez}
\address{Departamento de Matem\'aticas - Facultad de Ciencias B\'asicas - Universidad de Antofagasta - Chile}
\email{maria.alvarez@uantof.cl}

\author{Isabel Hern\'andez}
\address{CONACYT - CIMAT - Unidad M\'erida}
\email{ isabel@cimat.mx}

\maketitle

\begin{abstract}
In this paper we study the varieties of nilpotent Lie superalgebras of dimension $\leq 5$. We provide the algebraic classification of these superalgebras and obtain the irreducible components in every variety. As a byproduct we construct rigid nilpotent Lie superalgebras of arbitrary dimension.
\end{abstract}


\section{Introduction}
Nilpotent Lie superalgebras are an important class of Lie superalgebras and there are not many results concerning them in the literature. There are some papers concerning the classification of low-dimensional  nilpotent Lie superalgebras, in \cite{He} some nilpotent Lie superalgras are missing and in \cite{MF} the classification of some 
$(2|3)$-dimensional nilpotent Lie superalgebras includes three parametric families which are in fact not families (see remark \ref{cotejo2-3}).
On the other hand, some nilpotent Lie superalgebras of maximal nilindex were studied in \cite{GKN}, and in this process a rough classification of 
nilpotent Lie superalgebras of dimension (2|3) is given.

\medskip

The study of the geometric classification of algebras, their degenerations and their irreducible components is an active research field. Among all structures, we mention Lie algebras (see, for instance, \cite{A}, \cite{B1}, \cite{B2}, \cite{BS}, \cite{CD}, \cite{CPSW}, \cite{GO}, \cite{HNPT}, \cite{KN}, \cite{L}, \cite{NP}, \cite{S}, \cite{W}), Jordan algebras (see, for instance, \cite{ACGS}, \cite{AFM}, \cite{GKP}, \cite{KM}, \cite{KP}) and superstructures (see \cite{AZ}, \cite{AH}, \cite{FP} and \cite{AHK}).

\medskip

In this work, we provide both the algebraic and geometric classifications of nilpotent Lie superalgebras of dimension $\leq 5$.

\section{Preliminaries}
A  supervector space $V=V_0\oplus V_1$ over the field $\F$ is a $\Z_2$-graded  $\F$-vector space, i.e.,  a  vector space decomposed into a direct sum of two subspaces. The elements of
 $V_0\setminus\{ 0 \}$ (respectively,  on $ V_1 \setminus \{0\}$) are called even (respectively, odd). Even and odd elements together are called homogeneous; the degree of a homogeneous element is
 $i$, denoted by $|v|=i$,  if $v \in V_i \setminus \{ 0 \}$, for $i \in \Z_2$. If $\dim_\F (V_0)=m$ and $\dim_\F (V_1)=n$, we say that the dimension of $V$ is $(m|n)$.   The vector space $\End(V)$ can be viewed as a  supervector space,  denoted by  $\End(V_0|V_1)$,  where  $\End( V_0 | V_1)_i = \{  T  \in \End(V) \; | \; T(V_j)\subset V_{i+j},   \;  j \in \Z_2  \}$, for $i\in \Z_2$. Given a homogeneous basis  $\{e_1,  \dots, e_{m}, f_1, \dots, f_n\}$ for  $V=V_0 \oplus V_1$ (that is, $V_0 = \Span_\F\{ e_1, \dots,
 e_m\} $ and $V_1 = \Span_\F\{ f_1, \dots,
 f_n\}$), it follows that $\End(V_0 | V_1)_i$ can be identified with $(\Mat_{(m|n)}(\F))_i$, for $ i \in \Z_2$,  where
$$
\begin{array}{l}
(\Mat_{(m|n)}(\F))_0= \left\{
\begin{psmallmatrix}
A & 0 \\
0 & D
\end{psmallmatrix} |\, A \in \Mat_m(\F),  \; D \in \Mat_n(\F)   \right\}, \quad \text{ and }
 \\
 \\
(\Mat_{(m|n)}(\F))_1= \left\{
\begin{psmallmatrix}
0 & B \\
C & 0
\end{psmallmatrix} |\, C \in \Mat_{n \times m}(\F), \; B \in \Mat_{m\times n}(\F) \right\}.
 \end{array}
$$

In particular,  $\Aut(V_0|V_1)$ is a $\Z_2$-graded group such that $\Aut(V_0|V_1)_0$  can be identified  with   $\GL_m(\F) \oplus \GL_n(\F)$.

\medskip
A {\bf Lie superalgebra} is a supervector space $\g=\g_0 \oplus \g_1$ endowed with a
bilinear map  $\[ \cdot,  \cdot \] : \g \times  \g \to \g$   satisfying the following:
\begin{enumerate}[(i)]
\item $\[\g_i, \g_j\] \subset \g_{i+j}$, for $i, j \in \Z_2$.
\item Super skew-symmetry:   $\[x,y\]= -(-1)^{|x||y|}\[y, x\]$.
\item Super Jacobi identity:
 $$(-1)^{|x||z|}\[ \[ x,y    \], z    \] +(-1)^{|x||y|}\[ \[ y,z    \], x    \] +   (-1)^{|y||z|}\[ \[ z,x    \], y    \] =0$$
\end{enumerate}
for  $x,y,z \in (\g_0 \cup \g_1) \setminus \{ 0 \}$.  A linear map  between Lie superalgebras $\Phi: \g \to \g^\prime$  is called a Lie superalgebra morphism if $\Phi$ is even (i.e.  $\Phi(\g _i) \subset    \g_i $,  for $ i \in \Z_2$), and
$\Phi (\[ x,y\] )= \[ \Phi(x),  \Phi(y) \] ^\prime$, for $x,y \in \g$. 
(see \cite{Sc} for standard terminology on Lie superalgebras).

\medskip

\subsection{The variety $\LS$}
Let  $V=V_0\oplus V_1$ be a  complex $(m|n)$-dimensional  supervector space with a fixed homogeneous basis
$\left\{e_1,\dots,e_m, f_1,\dots,f_n\right\}$.  Given a Lie superalgebra structure $\[ \cdot, \cdot \]$ on $V$,  we can identify  $ \g=(V , \,  \[ \cdot, \cdot \])$  with its set of structure cons\-tants 
$\left\{c_{ij}^k,\rho_{ij}^k,\Gamma_{ij}^k\right\}\in\C^{m^3+2mn^2}$, where
$$\[e_i,e_j\]=\sum_{k=1}^mc_{ij}^ke_k,\quad \[e_i,f_j\]=\sum_{k=1}^n\rho_{ij}^kf_k,\quad\text{and}\quad \[f_i,f_j\]=\sum_{k=1}^m\Gamma_{ij}^ke_k.$$
Since every set of structure constants must satisfy the polynomial equations given by  super skew-symmetry and the super Jacobi identity,
the set of all Lie superalgebras of dimension $(m|n)$ is an algebraic variety in $\C^{m^3+2nm^2}$, denoted by $\LS$.  
Since Lie superalgebra isomorphisms are  even maps, it follows that  the group $G=\Aut(\g_0|\g_1)_0\simeq\GL_m(\C) \oplus  \GL_n(\C)$ acts on $\LS$ by ``change of basis'': 
$$g\cdot\[x,y\]=g\[g^{-1}x,g^{-1}y\],\quad\text{for $g\in G$, and $x,y\in \g$}.$$
Observe that the set of $G$-orbits of this action is in one-to-one correspondence with the isomorphism classes in $\LS$.

\medskip

Given two Lie superalgebras $\g$ and $\h$, we say that $\g$ {\bf degene\-rates} to $\h$, denoted by $\g\rightarrow\h$, if $\h$ lies in the Zariski closure of the $G$-orbit $O(\g)$. The process of degeneration induces a partial order in the orbit space of $\LS$, given by $O(\h)\leq O(\g)$ if and only if $ \h\in\overline{O(\g)}$. 

\medskip

Since each orbit $O(\g)$ is a constructible set, its closures relative to the Euclidean and the Zariski topologies are the same  (see \cite{M}, 1.10  Corollary 1, p. 84).  As a  consequence  the following  is obtained:

\begin{lemma}
Let $\C(t)$ be the field of fractions of the polynomial ring $\C[t]$. If there exists an operator $g_t\in\GL_m(\C(t))\oplus\GL_n(\C(t))$ such that $\displaystyle\lim_{t\rightarrow 0}g_t\cdot\g=\h$, then $\g\rightarrow\h$.
\label{def-equiv-defor}
\end{lemma}

 From the previous lemma,  it follows that  every Lie superalgebra $\g\in\LS$ degenerates to the Lie superalgebra $\mathfrak{a}=\{0,0,0\}$. In fact, take  $g_t=t^{-1}\left(\operatorname{id}_{m}\oplus \operatorname{id}_{n}\right)$, where $\operatorname{id}_{k}$ is the identity map in $\GL_{k}(\C)$;  then  $\displaystyle\lim_{t\rightarrow 0}g_t\cdot\g=\frak a$. Thus, $ \g \rightarrow  \mathfrak{a}$.

\begin{definition}\label{def:rigid}
An element $\g\in \LS$ is called {\bf rigid} if its orbit $O(\g)$ is open in $\LS$.
\end{definition}

Rigid elements of the variety are important due to the fact that if $\g$ is rigid in $\LS$, then there exists an irreducible component $C$ such that $C\cap O(\g)$ is a non-empty and open subset of $C$ and thus $C\subset\overline{O(\g)}$. 

\medskip

In order to prove the rigidity of Lie superalgebras, we will compute the cohomology group $(H^2(\g,\g))_0$ which parameterizes the infinitesimal deformations of the Lie superalgebra $\g$ (for details on deformations of Lie superalgebras see for instance \cite{B}).

\subsection{The variety $\NLS$} Let $\g=\g_0 \oplus \g_1$ be a Lie superalgebra. Define \begin{equation}
\g^0:=\g \quad \text{and} \quad \g^k:= \[ \g, \g^{k-1}\], \quad \text{for } k \geq 1. 
\label{nilpotentdefinition}
\end{equation} 
The superalgebra $\g$ is called {\bf nilpotent} if there exists $k \geq 1$ such that $\g^k=0$. 
We will denote by $\NLS$ the set of all  nilpotent Lie superalgebras of dimension $(m|n)$.  
Notice that $\NLS \subset \LS$ is an algebraic
subvariety, since 
the condition $\g^k=0$ for some $k\geq1$  
is determined by polynomial equations on the structure constants.

\subsection{Invariants}

In order to show the non-existence of degenerations, we need a series of invariants. First we recall some definitions:


\begin{definition}
Given a Lie superalgebra $\g=\g_0\oplus\g_1$, an $(\alpha,\beta,\gamma)$-derivation of degree $i$ of $\g$ is a linear map $D\in\End(\g_0|\g_1)_i$ such that
$$\alpha D(\[x,y\])=\beta\[D(x),y\]+(-1)^{i|x|}\gamma\[x,D(y)\].$$
\end{definition}

We denote the set of all $(\alpha,\beta,\gamma)$-derivations of degree $i$ of $\g$ by $(\mathfrak{D}_{\alpha,\beta,\gamma}(\g))_i$. For details on $(\alpha,\beta,\gamma)$-derivations of Lie superalgebras see \cite{ZZ}.

\medskip

We can identify $\g$ with the triple $\left([\cdot,\cdot],\rho,\Gamma\right)$, where $[\cdot,\cdot]$ is determined by the set of structure constants $\{c_{ij}^k\}$,  $\rho$ is determined by the set of structure constants $\{\rho_{ij}^k\}$, and $\Gamma$ is determined by the set of structure constants $\{\Gamma_{ij}^k\}$. Then we make the following definition:

\begin{definition}
Let $\g=\left([\cdot,\cdot],\rho,\Gamma\right)$ be a Lie superalgebra.
\begin{enumerate}
\item $\ab(\g)=\left(0,0,\Gamma\right)$ is the Lie superalgebra with the same underlying vector superspace than $\g$, but with trivial structure constants $c_{ij}^k$ and $\rho_{ij}^k$ ($[\cdot,\cdot]=0$ and $\rho=0$).
\item $\mathcal{F}(\g)=\left([\cdot,\cdot],\rho,0\right)$ is the Lie superalgebra with the same underlying vector superspace as $\g$, but with trivial structure constants $\Gamma_{ij}^k$ ($\Gamma=0$).
\end{enumerate}
\end{definition}

{
\begin{notation}
Denote by $\t(\g)$ the maximal dimension of a trivial subalgebra of $\g$, i.e.
$$\t(\g)=\max\{\dim\h\ |\ \h\ \text{is a subalgebra of } \g\ \text{and }\[\cdot,\cdot\]_\h=0\}.$$
\end{notation}
Now consider the set
$$\Delta_{\alpha\beta}=\left\{\g\in\LS\ \left|\ c_{ij}^k=0,\rho_{ir}^t=0,\Gamma_{rs}^k=0\ \text{for }
\begin{array}{l}
m-\alpha+1\leq i,j\leq m,\\
n-\beta+1\leq r,s,\leq n
\end{array}
\right.\right\}.$$
It is clear that $\g\in\Delta_{\alpha\beta}$ if and only if $\{e_{m-\alpha+1},\dots,e_m,f_{n-\beta+1},\dots,f_n\}$ is a trivial subalgebra of $\g$. Moreover, $(G\cdot\g)\cap\Delta_{\alpha\beta}\neq\emptyset$ if and only if $\t(\g)\geq\alpha+\beta$. Notice that $\Delta_{\alpha\beta}$ is a closed set but is not  $G$-stable.

\medskip

Let $B$ be a Borel subgroup of $G$ consisting of lower triangular matrices. Then the following result is obtained.

\begin{lemma}\label{lem:stable}
$\Delta_{\alpha\beta}$ is $B$-stable for every $1\leq\alpha\leq m$ and $1\leq \beta\leq n$.
\end{lemma}

\begin{proof}
Let $\g\in\Delta_{\alpha\beta}$, $b\in B$. If $m-\alpha+1\leq i,j\leq m$ and $n-\beta+1\leq r,s,\leq n$, then $b^{-1}(e_i),b^{-1}(e_j),b^{-1}(f_r),b^{-1}(f_s)\in\Span\{e_{m-\alpha+1},\dots,e_m\ |\ f_{n-\beta+1},\dots,f_n\}$ and thus
\begin{align*}
\[e_i,e_j\]_{b\cdot\g}=&b(\[b^{-1}(e_i),b^{-1}(e_j)\]_\g)=0,\\
\[e_i,f_r\]_{b\cdot\g}=&b(\[b^{-1}(e_i),b^{-1}(f_r)\]_\g)=0,\\
\[f_r,f_s\]_{b\cdot\g}=&b(\[b^{-1}(f_r),b^{-1}(f_s)\]_\g)=0.
\end{align*}
Therefore $(b\cdot\g)\in\Delta_{\alpha\beta}$.
\end{proof}

Next, we state an important result whose proof can be consulted in \cite{GO}, Proposition 1.17.

\begin{proposition}\label{prop:GB}
Let $G$ be a reductive algebraic group over $\C$ with Borel subgroup $B$ and let $X$ be an algebraic set on which $G$ acts rationally. For $x\in X$,
$$\overline{G\cdot x}= G\cdot\overline{B\cdot x}.$$
\end{proposition}
}

Next we summarize the list of relations that will be used.

\begin{lemma}\label{lem:invariants}
Let $\g$ and $\h$ be Lie superalgebras of dimension $(m|n)$. If $\g\rightarrow\h$ then the following relations must hold:
\begin{enumerate}
\item $\dim O(\g)>\dim O(\h)$.
\item\label{inv:gamma} If $\Gamma(\g)\equiv 0$ then $\Gamma(\h)\equiv 0$.
\item\label{inv:centro} $\dim\z(\g)_i\leq \dim\z(\h)_i$ for $i\in\Z_2$, where $\z(\g)$ is the center of $\g$.
\item\label{inv:der} $\dim\[\g,\g\]_i\geq\dim\[\h,\h\]_i$ for $i\in\Z_2$.
\item\label{inv:ab} $\ab(\g)\rightarrow\ab(\h)$.
\item\label{inv:f} $\mathcal{F}(\g)\rightarrow \mathcal{F}(\h)$.
\item\label{inv:dergen} $\dim\left(\mathfrak{D}_{\alpha,\beta,\gamma}(\g)\right )_i\leq\dim\left(\mathfrak{D}_{\alpha,\beta,\gamma}(\h) \right)_i$ for $i\in\Z_2$.
\item\label{inv:trivial} $\t(\g)\leq \t(\h)$.
\end{enumerate}
\end{lemma}

\begin{proof}
The proof for invariants (1)-(7) can be found in \cite{AH}, Lemmas 2.3, 2.4 and 2.5. We will prove (8).

Let $\g,\h\in\LS$ such that $\g\to\h$. Let $\mathfrak{k}$ be a trivial subalgebra of $\g$ of dimension $\alpha+\beta=\t(\g)$. Choose $g\in G$ such that $(g\cdot\g)\in\Delta_{\alpha\beta}$. By Lemma \ref{lem:stable}, $B\cdot(g\cdot\g)\in\Delta_{\alpha\beta}$. Since $\Delta_{\alpha\beta}$ is a closed set, it follows that $\overline{B\cdot(g\cdot\g)}\subset\Delta_{\alpha\beta}$.  By Proposition \ref{prop:GB}, $ \h\in \overline{G\cdot(g\cdot\g)}=G\cdot\overline{B\cdot(g\cdot\g)}\subset G\cdot\Delta_{\alpha\beta}$. Therefore, $\t(\g)\leq\t(\h)$.

\end{proof}

\medskip

\section{On the classification problem of  Lie superalgebras}

A Lie superalgebra  $\g=\g_0 \oplus \g_1$ can be identified with a triple ($[\cdot, \cdot ], \,  \rho,  \, \Gamma $), where
\begin{enumerate}[(i)]
\item  $[ \cdot, \cdot] =  \[\cdot, \cdot \]|_{ \g_0 \times \g_0}$ is a Lie bracket on $\g_0$. 
\item  $\rho: \g_0 \to \gl(\g_1)$  is a representation defined by  $\rho(x):= \[x, \cdot\]$, 
\item $\Gamma: = \[\cdot, \cdot \]|_{\g_1 \times \g_1} :   \g_1 \times \g_1 \to \g_0$  is a symmetric bilinear map
satisfying the following identities:
\end{enumerate}
\begin{itemize}
\item[{(J1)}] $[x, \Gamma(u,v)]= \Gamma(\rho(x)u, v) + \Gamma(u, \rho(x)v), \; \text{ for } \;  x \in \g_0 \text{ and } u,v\in \g_1$.
\item [(J2)] $  \rho(\Gamma(u,v))(w) +  \rho(\Gamma(v,w))(u) +  \rho(\Gamma(u,w))(v)=0,  \; \text{ for } \; u,v,w \in \g_1$.
\end{itemize}

Using this notation,  the action of the group $\GL_m  (\C) \oplus\GL_n(\C)$ on  $\LS$ can by written as:
\begin{equation}
(T,S)\cdot ([\cdot, \cdot ], \,   \rho,  \, \Gamma)=  ([\cdot, \cdot ]^\prime, \,   \rho^\prime,  \, \Gamma^\prime),
\label{action}
\end{equation}
where
\begin{eqnarray*}
\label{iso1}
[\cdot, \cdot]^\prime &=& T( [T^{-1}(\cdot), \;  T^{-1}(\cdot) ]), \\  
\label{iso2}  
\rho^\prime(\cdot)&=&S \circ \rho(T^{-1}(\cdot)) \circ S^{-1},\\  
\label{iso3}
\Gamma^\prime (\cdot, \cdot)&=&T(\Gamma( S^{-1}(\cdot),  S^{-1}(\cdot))).
\end{eqnarray*}

Hence, classifying Lie superalgebras is equivalent to finding the orbits of the action described in \eqref{action}. 

\medskip

On the other hand, notice that  a  Lie superalgebra $\g= \g_0 \oplus \g_1$ is nilpotent if and only if $\g_0$ is a nilpotent  Lie algebra and $\g_0$ acts on $\g_1$ by nilpotent endomorphisms.  The technique  used in this work for classifying  all possi\-ble structures of nilpotent Lie superalgebras of dimension $(m|n)$ with $m+n \leq 5$ consists {of} three steps{:
\begin{enumerate}
\item Describe the nilpotent $\g_0$-modules of dimension $n$, for every nilpotent Lie algebra $\g_0$.
\item Describe all possible  symmetric bilinear maps $\Gamma: \g_1 \times \g_1 \to
\g_0$ satisfying (J1) and (J2), where $\g_1$ is a  nilpotent Lie $\g_0$-module of dimension $n$.
\item Find representatives for the isomorphism classes of such Lie superalgebras. 
\end{enumerate}}

\subsection{Lie superalgebras for which $\g_0$ is  an abelian Lie algebra and the action of $\g_0$ on $\g_1$ is trivial.} This is an important type of nilpotent Lie superalgebra.
Under the hypotheses above, a Lie superalgebra  $\g$ is completely determined by  the symmetric bilinear map $\Gamma: \g_1 \times \g_1 \to \g_0$. Given bases $\{ e_1, \dots, e_m\}$ and 
$ \{f_1, \cdots, f_n \}$ for $\g_0$ and $\g_1$ respectively, we write $\Gamma(f_i, f_j)= \displaystyle\sum_{k=1}^m \Gamma_{ij}^k e_k $. Then  we can identify 
$$
\Gamma \longleftrightarrow ( \Gamma^1, \dots,  \Gamma^m ),
$$
where $\Gamma^k =  (\Gamma^k_{ij})$ is a symmetric matrix, for $k\in \{ 1, \dots, m\}$.  Hence,  under this identification, two Lie superalgebras $ ( \Gamma^1, \dots,  \Gamma^m )  $ and  $( {\Gamma^\prime}^1, \dots, {\Gamma^\prime}^m )$ belong to the same orbit under the action given in \eqref{action} if and only if there exist $T\in \GL_m(\C)$ and $S\in GL_n(\C)$ such that 
$$
{\Gamma^\prime}^k= \sum_{i=1}^m T_{ki}S^t\Gamma^iS, \quad  k = 1, \dots, m. 
$$
Notice that the problem of finding the orbits is considered a so-called  ``wild pro\-blem''.
In order to find representatives for the orbits it is useful to study the cases according to the number of simultaneously diagonalizable matrices in  $\{ \Gamma^1, \dots, \Gamma^m\}$. This number is invariant in each orbit. 
 To illustrate this technique, in Section 7.3 we classify the Lie superalgebras of dimension $(2|3)$ having $[\cdot, \cdot]=0$ and $\rho=0$. Our results  show that some $(2|3)$-nilpotent Lie superalgebras are missing in  \cite{He}  and in \cite{MF}.
 Moreover, we show that there are no infinite families  of mutually non-isomorphic  superalgebras, contrary to what is claimed in \cite{MF}.  

\medskip

\section{Nilpotent Lie superalgebras of dimension $\leq 4$}

In this section we summarize the algebraic and geometric classifications of Lie superalgebras of dimension $m+n\leq 4$. We provide, in every dimension, the rigid elements by computing the group $(H^2(\g,\g))_0$. Also, non-degeneration criteria are
obtained by applying Lemma \ref{lem:invariants}. After discarding possible degenerations, we give a list with all primary degenerations (those that cannot be obtained by transitivity). We exemplify these techniques in dimension 3. Finally, the irreducible components of every variety are obtained by looking at the Hasse diagram of degenerations of all superalgebras.

\subsection{Dimension 2}

\begin{theorem}
Nilpotent Lie superalgebras of dimension 2 are, up to isomorphism:
\begin{table}[h]
\begin{tabular}{lll}
$(2|0)_0$: & & $\[\cdot,\cdot\]=0.$\\
$(1|1)_0$: & & $\[\cdot,\cdot\]=0.$\\
$(1|1)_1$: & & $\[f_1,f_1\]=e_1.$\\
$(0|2)_0$: & & $\[\cdot,\cdot\]=0$.\\
\end{tabular}
\end{table}
\end{theorem}

\medskip

\begin{proposition}\label{prop:rigid2}
The Lie superalgebra $(1|1)_1$ is rigid in the variety $\mathcal{LS}^{(1|1)}$.
\end{proposition}

\begin{proof}
$(H^2((1|1)_1,(1|1)_1))_0=0$.
\end{proof}

\begin{theorem}\ 
\begin{enumerate}
\item The irreducible component of the variety $\mathcal{N}_{(2|0)}$ is $\mathcal{C}_1=\overline{O((2|0)_0)}$.
\item The irreducible component of the variety $\mathcal{N}_{(1|1)}$ is $\mathcal{C}_1=\overline{O((1|1)_1)}$.
\item The irreducible component of the variety $\mathcal{N}_{(0|2)}$ is $\mathcal{C}_1=\overline{O((0|2)_0)}$.
\end{enumerate}
\end{theorem}

\scriptsize
 
\begin{center}

\begin{tikzpicture}[->,>=stealth',shorten >=0.08cm,auto,node distance=1.25cm,
                    thick,main node/.style={rectangle,draw,fill=gray!12,rounded corners=1.5ex,font=\sffamily \bf \bfseries },
                    purple node/.style={rectangle,draw, color=purple,fill=gray!12,rounded corners=1.5ex,font=\sffamily \bf \bfseries },
                    orange node/.style={rectangle,draw, color=orange,fill=gray!12,rounded corners=1.5ex,font=\sffamily \bf \bfseries },
                    green node/.style={rectangle,draw, color=green,fill=gray!12,rounded corners=1.5ex,font=\sffamily \bf \bfseries },
                    olive node/.style={rectangle,draw, color=olive,fill=gray!12,rounded corners=1.5ex,font=\sffamily \bf \bfseries },
                    connecting node/.style={circle, draw, color=purple },
                    rigid node/.style={rectangle,draw,fill=black!20,rounded corners=1.5ex,font=\sffamily \tiny \bfseries },style={draw,font=\sffamily \scriptsize \bfseries }]
                    
\node (21)   {$\dim O(\g)$};

\node (11) [below          of=21]      {$1$};
\node (01) [below          of=11]      {$0$};
\node (12) [right          of=11]      {};
\node (13) [right          of=12]      {};
\node (14) [right          of=13]      {};
\node (02) [right          of=01]      {};
\node (03) [right          of=02]      {};
\node (04) [right          of=03]      {};

\node [main node] (111)  [right of =12]                      {$(1|1)_1$ };
\node [main node] (200)  [right of =01]                      {$(2|0)_0$ };
\node [main node] (110)  [right of =02]                      {$(1|1)_0$ };
\node [main node] (020)  [right of =03]                      {$(0|2)_0$ };

\path[every node/.style={font=\sffamily\small}]

(111)  edge [bend right=0, color=black] node{}  (110);

\end{tikzpicture}

\end{center}

\normalsize

\subsection{Dimension 3}

\begin{theorem}
Nilpotent Lie superalgebras of dimension 3 are, up to isomorphism:
\begin{table}[h]
\begin{centering}
\begin{tabular}{llll}
$(3|0)_0$: & & $\[\cdot,\cdot\]=0.$ &\\
$(3|0)_1$: & & $\[e_1,e_2\]=e_3.$ &\\
$(2|1)_0$: & & $\[\cdot,\cdot\]=0.$ &\\
$(2|1)_1$: & & $\[f_1,f_1\]=e_1.$ & \\
$(1|2)_0$: & & $\[\cdot,\cdot\]=0.$ & \\
$(1|2)_1$: & & $\[f_1,f_1\]=e_1.$ & \\
$(1|2)_2$: & & $\[f_1,f_2\]=e_1.$ & \\
$(1|2)_3$: & & $\[e_1,f_2\]=f_1.$ & \\
$(0|3)_0$: & & $\[\cdot,\cdot\]=0.$ &\\
\end{tabular}
\end{centering}
\end{table}
\end{theorem}

\medskip

\begin{longtable}{|c|c|}
\caption[]{Non-degenerations in dimension 3}
\label{table:non-deg}\\
\hline
$\g\not\to\h$ & Reason\\
\hline
\endfirsthead
\caption[]{(continued)}\\
\hline
$\g\not\to\h$ & Reason\\
\hline
\endhead
$(1|2)_2\not\to(1|2)_3$ & Lemma \ref{lem:invariants} (\ref{inv:der}) $i=1$\\ \hline\hline
\end{longtable}

\medskip

In this case $(1|2)_2\not\to(1|2)_3$ since $\dim\[(1|2)_2,(1|2)_2\]_1=0<1=\dim\[(1|2)_3,(1|2)_3\]_1$, contradicting Lemma \ref{lem:invariants} (\ref{inv:der}).

\medskip

\begin{longtable}{|c|lll|}
\caption[]{Degenerations in dimension 3}
\label{table:deg}\\
\hline
$\g\to\h$ & \multicolumn{3}{|c|}{Parametrized Basis}\\
\hline
\endfirsthead
\caption[]{(continued)}\\
\hline
$\g\to\h$ & \multicolumn{3}{|c|}{Parametrized Basis}\\
\hline
\endhead
$(1|2)_2\to(1|2)_1$ & $x_1=e_1$, & $y_1=f_1+\frac{1}{2}f_2$, & $y_2=tf_2$. \\ \hline
\end{longtable}

Consider the base change of $(1|2)_2$ given by
$$x_1=e_1,\quad y_1=f_1+\frac{1}{2}f_2\quad\text{and}\quad y_2=tf_2.$$
In terms of this basis, the non-zero barckets are:
$$\[y_1,y_1\]=x_1\quad\text{and}\quad\[y_1,y_2\]=tx_1.$$
By letting $t\to 0$, we obtain the Lie product of $(1|2)_1$. Therefore, $(1|2)_2\to(1|2)_1$.

\medskip

\begin{proposition}\label{prop:rigid3}\ 
\begin{itemize}
\item The Lie superalgebra $(2|1)_1$ is rigid in the variety $\mathcal{N}_{(2|1)}$.
\item The Lie superalgebra $(1|2)_3$ is rigid in the variety $\mathcal{N}_{(1|2)}$.
\item The Lie superalgebra $(1|2)_2$ is rigid in the variety $\mathcal{LS}_{(1|2)}$.
\end{itemize}
\end{proposition}

\begin{proof}
\begin{align*}
(H^2((2|1)_1,(2|1)_1))_0 = &\Span\{ -2e_1^*\wedge e_2^*\otimes e_1+e_2^*\wedge f_1^*\otimes f_1\},\\ 
(H^2((1|2)_3,(1|2)_3))_0 = &\Span\{ e_1^*\wedge f_1^*\otimes  f_1,e_1^*\wedge f_1^*\otimes f_2\},\\ 
(H^2((1|2)_2,(1|2)_2))_0 = &\,0.
\end{align*}
Notice that for $(2|1)_1$, every possible deformation provides a non-nilpotent Lie superalgebra. For $(1|2)_3$ the only cocycle that could give a nilpotent deformation is $e_1^*\wedge f_1^*\otimes f_2$, but in this case we obtain that the product of the deformed Lie superalgebra $\g_t$ is given by
$$\[e_1,f_2\]_t=f_1,\qquad\[e_1,f_1\]_t=tf_2.$$
If $t\neq 0$, then $(\g_t)^k=(\g_t)^1=\Span\{ f_1,f_2\}$ for every $k>1$, and therefore $\g_t$ is not nilpotent.
\end{proof}

\medskip

\begin{theorem}\ 
\begin{enumerate}
\item The irreducible component of the variety $\mathcal{N}_{(3|0)}$ is $\mathcal{C}_1=\overline{O((3|0)_1)}$.
\item The irreducible component of the variety $\mathcal{N}_{(2|1)}$ is $\mathcal{C}_1=\overline{O((2|1)_1)}$.
\item The irreducible components of the variety $\mathcal{N}_{(1|2)}$ are:
\begin{enumerate}
\item $\mathcal{C}_1=\overline{O((1|2)_2)}$.
\item $\mathcal{C}_2=\overline{O((1|2)_3)}$.
\end{enumerate}
\item The irreducible component of the variety $\mathcal{N}_{(0|3)}$ is $\mathcal{C}_1=\overline{O((0|3)_0)}$.
\end{enumerate}
\end{theorem}

\medskip

\scriptsize

\begin{center}

\begin{tikzpicture}[->,>=stealth',shorten >=0.08cm,auto,node distance=1.25cm,
                    thick,main node/.style={rectangle,draw,fill=gray!12,rounded corners=1.5ex,font=\sffamily \bf \bfseries },
                    purple node/.style={rectangle,draw, color=purple,fill=gray!12,rounded corners=1.5ex,font=\sffamily \bf \bfseries },
                    orange node/.style={rectangle,draw, color=orange,fill=gray!12,rounded corners=1.5ex,font=\sffamily \bf \bfseries },
                    green node/.style={rectangle,draw, color=green,fill=gray!12,rounded corners=1.5ex,font=\sffamily \bf \bfseries },
                    olive node/.style={rectangle,draw, color=olive,fill=gray!12,rounded corners=1.5ex,font=\sffamily \bf \bfseries },
                    connecting node/.style={circle, draw, color=purple },
                    rigid node/.style={rectangle,draw,fill=black!20,rounded corners=1.5ex,font=\sffamily \tiny \bfseries },style={draw,font=\sffamily \scriptsize \bfseries }]
                    
\node (41)   {$\dim O(\g)$};

\node (31) [below          of=41]      {$3$};
\node (21) [below          of=31]      {$2$};
\node (11) [below          of=21]      {$1$};
\node (01) [below          of=11]      {$0$};

\node (32) [right          of=31]      {};
\node (33) [right          of=32]      {};
\node (34) [right          of=33]      {};
\node (35) [right          of=34]      {};
\node (22) [right          of=21]      {};
\node (23) [right          of=22]      {};
\node (24) [right          of=23]      {};
\node (25) [right          of=24]      {};
\node (12) [right          of=11]      {};
\node (13) [right          of=12]      {};
\node (14) [right          of=13]      {};
\node (15) [right          of=14]      {};
\node (16) [right          of=15]      {};
\node (02) [right          of=01]      {};
\node (03) [right          of=02]      {};
\node (04) [right          of=03]      {};
\node (05) [right          of=04]      {};
\node (06) [right          of=05]      {};

\node [main node] (301)  [right of =31]                      {$(3|0)_1$ };
\node [main node] (211)  [right of =22]                      {$(2|1)_1$ };\node [main node] (122)  [right of =34]                      {$(1|2)_2$ };
\node [main node] (121)  [right of =23]                      {$(1|2)_1$ };
\node [main node] (123)  [right of =25]                      {$(1|2)_3$ };
\node [main node] (300)  [right of =01]                      {$(3|0)_0$ };
\node [main node] (210)  [right of =02]                      {$(2|1)_0$ };
\node [main node] (120)  [right of =04]                      {$(1|2)_0$ };
\node [main node] (030)  [right of =06]                      {$(0|3)_0$ };

\path[every node/.style={font=\sffamily\small}]

(301)  edge [bend right=0, color=black] node{}  (300)
(211)  edge [bend right=0, color=black] node{}  (210)
(122)  edge [bend right=0, color=black] node{}  (121)
(121)  edge [bend right=0, color=black] node{}  (120)
(123)  edge [bend right=0, color=black] node{}  (120);

\end{tikzpicture}

\end{center}

\normalsize

\subsection{Dimension 4}

\begin{theorem}
Nilpotent Lie superalgebras of dimension 4 are, up to isomorphism:
\begin{table}[h]
\begin{centering}
\begin{tabular}{lllll}
$(4|0)_0$: & & $\[\cdot,\cdot\]=0.$\\
$(4|0)_1$: & & $\[e_1,e_2\]=e_3.$\\
$(4|0)_2$: & & $\[e_1,e_2\]=e_3$, & & $\[e_1,e_3\]=e_4.$\\
$(3|1)_0$: & & $\[\cdot,\cdot\]=0.$\\
$(3|1)_1$: & & $\[f_1,f_1\]=e_1.$\\
$(3|1)_2$: & & $\[e_1,e_2\]=e_3.$\\
$(3|1)_3$: & & $\[e_1,e_2\]=e_3$, & & $\[f_1,f_1\]=e_3.$\\
$(2|2)_0$: & & $\[\cdot,\cdot\]=0.$\\
$(2|2)_1$: & & $\[f_1,f_1\]=e_1$, & & $\[f_2,f_2\]=e_2.$\\
$(2|2)_2$: & & $\[f_1,f_1\]=e_1$, & & $\[f_2,f_2\]=e_1.$\\
$(2|2)_3$: & & $\[f_1,f_1\]=e_1.$\\
$(2|2)_4$: & & $\[f_1,f_2\]=e_1$, & & $\[f_2,f_2\]=e_2.$\\
$(2|2)_5$: & & $\[e_2,f_2\]=f_1.$\\
$(2|2)_6$: & & $\[e_2,f_2\]=f_1$, & & $\[f_2,f_2\]=e_1.$\\
$(1|3)_0$: & & $\[\cdot,\cdot\]=0.$\\
$(1|3)_1$: & & $\[e_1,f_2\]=f_1$, & & $\[e_1,f_3\]=f_2.$\\
$(1|3)_2$: & & $\[e_1,f_2\]=f_1.$\\
$(1|3)_3$: & & $\[f_1,f_1\]=e_1.$\\
$(1|3)_4$: & & $\[f_1,f_2\]=e_1.$\\
$(1|3)_5$: & & $\[f_1,f_1\]=e_1$, & & $\[f_2,f_3\]=e_1.$\\
$(0|4)_0$: & & $\[\cdot,\cdot\]=0.$\\
\end{tabular}
\end{centering}
\end{table}
\end{theorem}

\medskip

\begin{longtable}{|c|c|}
\caption[]{Non-degenerations in dimension 4}
\label{table:non-deg}\\
\hline
$\g\not\to\h$ & Reason\\
\hline
\endfirsthead
\caption[]{(continued)}\\
\hline
$\g\not\to\h$ & Reason\\
\hline
\endhead
$(2|2)_1\not\to(2|2)_5,(2|2)_6;\ (2|2)_4\not\to(2|2)_5,(2|2)_6$ & Lemma \ref{lem:invariants} (\ref{inv:f})\\ \hline\hline
$(1|3)_1\not\to(1|3)_3,(1|3)_4,(1|3)_5;\ (1|3)_2\not\to(1|3)_3,(1|3)_4,(1|3)_5$ & Lemma \ref{lem:invariants} (\ref{inv:gamma})\\ \hline
\end{longtable}

\medskip

\scriptsize{
\begin{longtable}{|c|llll|}
\caption[]{Degenerations dimension 4}
\label{table:deg}\\
\hline
$\g\to\h$ & \multicolumn{4}{|c|}{Parametrized Basis}\\
\hline
\endfirsthead
\caption[]{(continued)}\\
\hline
$\g\to\h$ & \multicolumn{4}{|c|}{Parametrized Basis}\\
\hline
\endhead
$(3|1)_3\to(3|1)_2$ & $x_1=e_1$, & $x_2=e_2$, & $x_3=e_3$, & $y_1=tf_1$. \\ \hline
$(3|1)_3\to(3|1)_1$ & $x_1=te_1$, & $x_2=e_2$, & $x_3=e_3$, & $y_1=f_1$. \\ \hline\hline
$(2|2)_1\to(2|2)_4$ & $x_1=t^{1/2}(-e_1+e_2)$, & $x_2=e_1+e_2$, & $y_1=t^{1/2}(-f_1+f_2)$, & $y_2=f_1+f_2$. \\ \hline
$(2|2)_4\to(2|2)_2$ & $x_1=e_1$, & $x_2=e_2$, & $y_1=-\frac{it^{-1}}{2}f_1+itf_2$, & $y_2=\frac{t^{-1}}{2}f_1+tf_2$. \\ \hline
$(2|2)_2\to(2|2)_3$ & $x_1=e_1$, & $x_2=e_2$, & $y_1=f_1$, & $y_2=tf_2$. \\ \hline
$(2|2)_6\to(2|2)_3$ & $x_1=e_1$, & $x_2=te_2$, & $y_1=f_2$, & $y_2=f_1$. \\ \hline
$(2|2)_6\to(2|2)_5$ & $x_1=t^{-1}e_1$, & $x_2=e_2$, & $y_1=f_1$, & $y_2=f_2$. \\ \hline\hline
$(1|3)_1\to(1|3)_2$ & $x_1=e_1$, & $y_1=f_1$, & $y_2=f_2$, & $y_3=tf_3$. \\ \hline
$(1|3)_5\to(1|3)_3$ & $x_1=e_1$, & $y_1=f_1$, & $y_2=tf_2$, & $y_3=f_3$. \\ \hline
$(1|3)_5\to(1|3)_4$ & $x_1=e_1$, & $y_1=f_3$, & $y_2=f_2$, & $y_3=tf_1$. \\ \hline
\end{longtable}
}

\normalsize
\medskip

\begin{proposition}\label{prop:rigid4}\ 
\begin{itemize}
\item The Lie superalgebra $(3|1)_3$ is rigid in the variety $\mathcal{N}_{(3|1)}$.
\item The Lie superalgebra $(2|2)_6$ is rigid in the variety $\mathcal{N}_{(2|2)}$.
\item The Lie superalgebra $(1|3)_1$ is rigid in the variety $\mathcal{N}_{(1|3)}$.
\item The Lie superalgebra $(2|2)_1$ is rigid in the variety $\mathcal{LS}_{(2|2)}$.
\item The Lie superalgebra $(1|3)_5$ is rigid in the variety $\mathcal{LS}_{(1|3)}$.
\end{itemize}
\end{proposition}

\begin{proof}
\begin{align*}
(H^2((3|1)_3,(3|1)_3))_0=&\Span\{ e_1^*\wedge e_2^*\otimes e_1, \; e_1^*\wedge e_2^*\otimes e_2, \;f_1^*\wedge f_1^*\otimes e_1-\frac{1}{2}e_2^*\wedge f_1^*\otimes f_1,\\
&\hspace{0.2cm}f_1^*\wedge f_1^*\otimes e_2+\frac{1}{2}e_1^*\wedge f_1^*\otimes f_1\},\\
(H^2((2|2)_1,(2|2)_1))_0=&0,\\
(H^2((2|2)_6,(2|2)_6))_0=&\Span\{ {e_2^*\wedge f_1^*\otimes f_2-f_1^*\wedge f_1^*\otimes e_1}, \;e_2^*\wedge f_2^*\otimes f_2-2e_1^*\wedge e_2^*\otimes e_1,\\
&\hspace{0.2cm}e_2^*\wedge f_1^*\otimes f_1, \;f_1^*\wedge f_2^*\otimes e_2-2e_1^*\wedge e_2^*\otimes e_2-2e_1^*\wedge f_1^*\otimes f_1\}\\
(H^2((1|3)_1,(1|3)_1))_0=&\Span\{ e_1^*\wedge f_1^*\otimes f_1,  \; e_1^*\wedge f_1^*\otimes f_2, \; e_1^*\wedge f_1^*\otimes f_3,e_1^*\wedge f_3^*\otimes f_3\}\\
(H^2((1|3)_5,(1|3)_5))_0=&0.
\end{align*}
Notice that for $(3|1)_3$ and $(1|3)_1$, every possible deformation provides non-nilpotent Lie superalgebras. For $(2|2)_6$, the only cocycle that could give a nilpotent Lie superalgebra is $e_2^*\wedge f_1^*\otimes f_2-f_1^*\wedge f_1^*\otimes e_1$, but in this case we obtain that the deformed Lie superalgebra $\g_t$ has product given by
$$\[e_2,f_2\]_t=f_1,\qquad\[f_2,f_2\]_t=e_1,\qquad\[e_2,f_1\]_t=tf_2,\qquad\[f_1,f_1\]_t=-te_1.$$
Then one can check that if $t\neq0$, then $(\g_t)^k=(\g_t)^1=\Span\{ e_1,f_1,f_2\}$  for every $k>1$, and thus $\g_t$ is not nilpotent.
\end{proof}

\newpage

\begin{theorem}\ 
\begin{enumerate}
\item The irreducible component of the variety $\mathcal{N}_{(4|0)}$ is $\mathcal{C}_1=\overline{O((4|0)_2)}$.
\item The irreducible component of the variety $\mathcal{N}_{(3|1)}$ is $\mathcal{C}_1=\overline{O((3|1)_3)}$.
\item The irreducible components of the variety $\mathcal{N}_{(2|2)}$ are:
\begin{enumerate}
\item $\mathcal{C}_1=\overline{O((2|2)_1)}$.
\item $\mathcal{C}_2=\overline{O((2|2)_6)}$.
\end{enumerate} 
\item The irreducible components of the variety $\mathcal{N}_{(1|3)}$ are:
\begin{enumerate}
\item $\mathcal{C}_1=\overline{O((1|3)_1)}$.
\item $\mathcal{C}_2=\overline{O((1|3)_5)}$.
\end{enumerate} 
\item The irreducible component of the variety $\mathcal{N}_{(0|4)}$ is $\mathcal{C}_1=\overline{O((0|4)_2)}$.
\end{enumerate}
\end{theorem}

\medskip

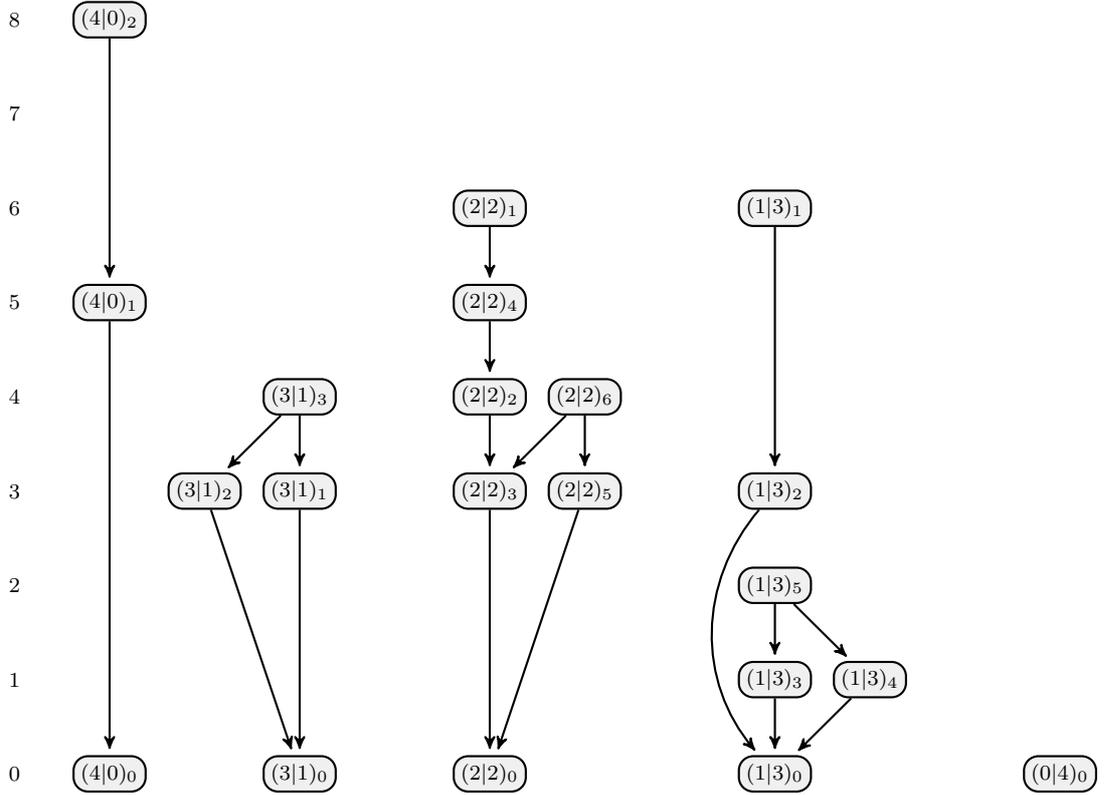
\scriptsize

\begin{center}

\begin{tikzpicture}[->,>=stealth',shorten >=0.08cm,auto,node distance=1.25cm,
                    thick,main node/.style={rectangle,draw,fill=gray!12,rounded corners=1.5ex,font=\sffamily \bf \bfseries },
                    purple node/.style={rectangle,draw, color=purple,fill=gray!12,rounded corners=1.5ex,font=\sffamily \bf \bfseries },
                    orange node/.style={rectangle,draw, color=orange,fill=gray!12,rounded corners=1.5ex,font=\sffamily \bf \bfseries },
                    green node/.style={rectangle,draw, color=green,fill=gray!12,rounded corners=1.5ex,font=\sffamily \bf \bfseries },
                    olive node/.style={rectangle,draw, color=olive,fill=gray!12,rounded corners=1.5ex,font=\sffamily \bf \bfseries },
                    connecting node/.style={circle, draw, color=purple },
                    rigid node/.style={rectangle,draw,fill=black!20,rounded corners=1.5ex,font=\sffamily \tiny \bfseries },style={draw,font=\sffamily \scriptsize \bfseries }]
                    
\node (91)   {$\dim O(\g)$};

\node (81) [below          of=91]      {$8$};
\node (71) [below          of=81]      {$7$};
\node (61) [below          of=71]      {$6$};
\node (51) [below          of=61]      {$5$};
\node (41) [below          of=51]      {$4$};
\node (31) [below          of=41]      {$3$};
\node (21) [below          of=31]      {$2$};
\node (11) [below          of=21]      {$1$};
\node (01) [below          of=11]      {$0$};

\node (82) [right          of=81]      {};
\node (62) [right          of=61]      {};
\node (63) [right          of=62]      {};
\node (64) [right          of=63]      {};
\node (65) [right          of=64]      {};
\node (66) [right          of=65]      {};
\node (67) [right          of=66]      {};
\node (68) [right          of=67]      {};
\node (52) [right          of=51]      {};
\node (53) [right          of=52]      {};
\node (54) [right          of=53]      {};
\node (55) [right          of=54]      {};
\node (42) [right          of=41]      {};
\node (43) [right          of=42]      {};
\node (44) [right          of=43]      {};
\node (45) [right          of=44]      {};
\node (46) [right          of=45]      {};
\node (32) [right          of=31]      {};
\node (33) [right          of=32]      {};
\node (34) [right          of=33]      {};
\node (35) [right          of=34]      {};
\node (36) [right          of=35]      {};
\node (37) [right          of=36]      {};
\node (38) [right          of=37]      {};
\node (22) [right          of=21]      {};
\node (23) [right          of=22]      {};
\node (24) [right          of=23]      {};
\node (25) [right          of=24]      {};
\node (26) [right          of=25]      {};
\node (27) [right          of=26]      {};
\node (28) [right          of=27]      {};
\node (12) [right          of=11]      {};
\node (13) [right          of=12]      {};
\node (14) [right          of=13]      {};
\node (15) [right          of=14]      {};
\node (16) [right          of=15]      {};
\node (17) [right          of=16]      {};
\node (18) [right          of=17]      {};
\node (19) [right          of=18]      {};
\node (02) [right          of=01]      {};
\node (03) [right          of=02]      {};
\node (04) [right          of=03]      {};
\node (05) [right          of=04]      {};
\node (06) [right          of=05]      {};
\node (07) [right          of=06]      {};
\node (08) [right          of=07]      {};
\node (09) [right          of=08]      {};
\node (010) [right          of=09]      {};
\node (011) [right          of=010]      {};

\node [main node] (402)  [right of =81]                      {$(4|0)_2$ };
\node [main node] (401)  [right of =51]                      {$(4|0)_1$ };
\node [main node] (400)  [right of =01]                      {$(4|0)_0$ };

\node [main node] (313)  [right of =43]                      {$(3|1)_3$ };
\node [main node] (312)  [right of =32]                      {$(3|1)_2$ };
\node [main node] (311)  [right of =33]                      {$(3|1)_1$ };
\node [main node] (310)  [right of =03]                      {$(3|1)_0$ };

\node [main node] (221)  [right of =65]                      {$(2|2)_1$ };
\node [main node] (224)  [right of =55]                      {$(2|2)_4$ };
\node [main node] (222)  [right of =45]                      {$(2|2)_2$ };
\node [main node] (226)  [right of =46]                      {$(2|2)_6$ };
\node [main node] (223)  [right of =35]                      {$(2|2)_3$ };
\node [main node] (225)  [right of =36]                      {$(2|2)_5$ };
\node [main node] (220)  [right of =05]                      {$(2|2)_0$ };

\node [main node] (131)  [right of =68]                      {$(1|3)_1$ };
\node [main node] (132)  [right of =38]                      {$(1|3)_2$ };
\node [main node] (135)  [right of =28]                      {$(1|3)_5$ };
\node [main node] (133)  [right of =18]                      {$(1|3)_3$ };
\node [main node] (134)  [right of =19]                      {$(1|3)_4$ };
\node [main node] (130)  [right of =08]                      {$(1|3)_0$ };

\node [main node] (040)  [right of =011]                      {$(0|4)_0$ };

\path[every node/.style={font=\sffamily\small}]

(402)  edge [bend right=0, color=black] node{}  (401)
(401)  edge [bend right=0, color=black] node{}  (400)

(313)  edge [bend right=0, color=black] node{}  (312)
(313)  edge [bend right=0, color=black] node{}  (311)
(312)  edge [bend right=0, color=black] node{}  (310)
(311)  edge [bend right=0, color=black] node{}  (310)

(221)  edge [bend right=0, color=black] node{}  (224)
(224)  edge [bend right=0, color=black] node{}  (222)
(222)  edge [bend right=0, color=black] node{}  (223)
(226)  edge [bend right=0, color=black] node{}  (223)
(226)  edge [bend right=0, color=black] node{}  (225)
(223)  edge [bend right=0, color=black] node{}  (220)
(225)  edge [bend right=0, color=black] node{}  (220)

(131)  edge [bend right=0, color=black] node{}  (132)
(132)  edge [bend right=40, color=black] node{}  (130)
(135)  edge [bend right=0, color=black] node{}  (133)
(135)  edge [bend right=0, color=black] node{}  (134)
(133)  edge [bend right=0, color=black] node{}  (130)
(134)  edge [bend right=0, color=black] node{}  (130);

\end{tikzpicture}

\end{center}

\normalsize

\medskip

\section{Nilpotent Lie superalgebras of Dimension 5}

We will study nilpotent Lie superalgebras of dimension 5 by separating them in three cases, namely:
\begin{itemize}
\item Case I: nilpotent Lie superalgebras of dimension $(5|0)$, $(4|1)$, $(1|4)$ and $(0|5)$.
\item Case II: nilpotent Lie superalgebras of dimension $(3|2)$.
\item Case III: nilpotent Lie superalgebras of dimension $(2|3)$.
\end{itemize} 

\subsection{Case I}
\begin{theorem}
Nilpotent Lie superalgebras of dimension $(m|n)$ with $m+n=5$ and $|m-n|>1$ are, up to isomorphism:
$$
\begin{array}{lllll}
(5|0)_0: &  \[\cdot, \cdot \]=0.&&& \\ 
(5|0)_1: & \[e_1,e_2\]=e_3. & && \\
(5|0)_2: & \[e_1,e_2\]=e_3, & \[e_1,e_3\]=e_4.&&\\
(5|0)_3: & \[e_1,e_2\]=e_3, & \[e_1,e_3\]=e_4,&\[ e_1, e_4\]=e_5, &\[ e_2, e_3\] =e_5.\\
(5|0)_4: & \[e_1,e_2\]=e_3, & \[e_1,e_3\]=e_4,&\[ e_1, e_4\]=e_5. &\\
(5|0)_5: & \[e_1,e_2\]=e_3, & \[e_1,e_4\]=e_5,&\[ e_2, e_3\] =e_5.&\\
(5|0)_6: & \[e_1,e_2\]=e_3, & \[e_1,e_3\]=e_4,&\[ e_2, e_3\]=e_5. &\\
(5|0)_7: & \[e_1,e_2\]=e_5, & \[e_3,e_4\]=e_5.& &\\
(5|0)_8: & \[e_1,e_2\]=e_4, & \[e_1,e_3\]=e_5.& &\\
(4|1)_0: &  \[\cdot, \cdot \]=0.&&& \\ 
(4|1)_1: &  \[f_1, f_1 \]=e_1.&&& \\ 
(4|1)_2: &  \[e_1, e_2\]=e_3, &&& \\ 
(4|1)_3: &  \[e_1, e_2\]=e_3, &\[f_1, f_1 \] = e_3.&& \\ 
(4|1)_4: &  \[e_1, e_2\]=e_3, &\[f_1, f_1 \] = e_4.&& \\ 
(4|1)_5: &  \[e_1, e_2\]=e_3, & \[ e_1, e_3\]=e_4.  && \\ 
(4|1)_6: &  \[e_1, e_2\]=e_3, & \[ e_1, e_3\]=e_4,  &\[f_1, f_1 \] = e_4.& \\ 
(1|4)_0: & \[\cdot, \cdot \]=0.&&& \\ 
(1|4)_1: & \[f_1, f_1 \]=e_1. & & &  \\   
(1|4)_2: & \[f_1, f_1 \]=e_1,&  \[f_2, f_2 \]=e_1. & & \\    
(1|4)_3: & \[f_1, f_1 \]=e_1,&  \[f_2, f_2 \]=e_1, &  \[f_3, f_3 \]=e_1. &  \\   
(1|4)_4: & \[f_1, f_1 \]=e_1,&  \[f_2, f_2 \]=e_1, &  \[f_3, f_3 \]=e_1, &  \[f_4, f_4 \]=e_1. \\    
(1|4)_5: & \[e_1, f_2 \]=f_1.&&& \\
(1|4)_6: & \[e_1, f_2 \]=f_1, &\[e_1, f_3 \]=f_2.  && \\    
(1|4)_7: & \[e_1, f_2 \]=f_1, &\[e_1, f_3 \]=f_2,  &\[e_1, f_4 \]=f_3. & \\   
(1|4)_8: & \[e_1, f_2 \]=f_1, &\[e_1, f_4 \]=f_3.& & \\   
(0|5)_0: & \[\cdot,\cdot\]=0. & & & \\   
\end{array}
$$

\end{theorem}

\begin{longtable}{|c|c|}
\caption[]{Non-degenerations Case I}
\label{table:non-deg}\\
\hline
$\g\not\to\h$ & Reason\\
\hline
\endfirsthead
\caption[]{(continued)}\\
\hline
$\g\not\to\h$ & Reason\\
\hline
\endhead
$(4|1)_2\not\to(4|1)_1;\ (4|1)_5\not\to(4|1)_4,(4|1)_3,(4|1)_1$ & Lemma \ref{lem:invariants} (\ref{inv:gamma})\\ \hline\hline
$(1|4)_{5}\not\to(1|4)_1;\ (1|4)_{8}\not\to(1|4)_2,(1|4)_1$ & Lemma \ref{lem:invariants} (\ref{inv:gamma})\\ \hline
$(1|4)_{2}\not\to(1|4)_5;\ (1|4)_{3}\not\to(1|4)_8,(1|4)_5;\ (1|4)_{4}\not\to(1|4)_8,(1|4)_5$ & Lemma \ref{lem:invariants} (\ref{inv:der}) $i=1$\\ \hline
$(1|4)_{6}\not\to(1|4)_3,(1|4)_2,(1|4)_1;\ (1|4)_{7}\not\to(1|4)_4,(1|4)_3,(1|4)_2,(1|4)_1$ & Lemma \ref{lem:invariants} (\ref{inv:der}) $i=0$\\ \hline
\end{longtable}


\begin{longtable}{|c|lllll|c|}
\caption[]{Degenerations Case I}
\label{table:deg}\\
\hline
$\g\to\h$ & \multicolumn{5}{|c|}{Parametrized Basis}\\
\hline
\endfirsthead
\caption[]{(continued)}\\
\hline
$\g\to\h$ & \multicolumn{5}{|c|}{Parametrized Basis}\\
\hline
\endhead
$(4|1)_3\to(4|1)_2$ & $x_1=e_1$, & $x_2=e_2$, & $x_3=e_3$, & $x_4=e_4$, & $y_1=tf_1$. \\ \hline
$(4|1)_3\to(4|1)_1$ & $x_1=e_3$, & $x_2=te_2$, & $x_3=e_1$, & $x_4=e_4$, & $y_1=f_1$. \\ \hline
$(4|1)_4\to(4|1)_3$ & $x_1=e_1$, & $x_2=e_2$, & $x_3=e_3$, & $x_4=t^{-1}(e_4-e_3)$, & $y_1=f_1$. \\ \hline
$(4|1)_5\to(4|1)_2$ & $x_1=e_1$, & $x_2=e_2$, & $x_3=e_3$, & $x_4=t^{-1}e_4$, & $y_1=f_1$.\\ \hline
$(4|1)_6\to(4|1)_5$ & $x_1=e_1$, & $x_2=e_2$, & $x_3=e_3$, & $x_4=e_4$, & $y_1=tf_1$. \\ \hline
$(4|1)_6\to(4|1)_4$ & $x_1=te_1$, & $x_2=t^{-1}e_2$, & $x_3=e_3$, & $x_4=e_4$, & $y_1=f_1$. \\ \hline\hline
$(1|4)_{2}\to(1|4)_{1}$ & $x_1=e_1$, & $y_1=f_1$, & $y_2=tf_2$, & $y_3=f_3$, & $y_4=f_4$. \\ \hline
$(1|4)_{8}\to(1|4)_{5}$ & $x_1=e_1$, & $y_1=f_1$, & $y_2=f_2$, & $y_3=f_3$, & $y_4=tf_4$. \\ \hline
$(1|4)_{3}\to(1|4)_{2}$ & $x_1=e_1$, & $y_1=f_1$, & $y_2=f_2$, & $y_3=tf_3$, & $y_4=f_4$. \\ \hline
$(1|4)_{4}\to(1|4)_{3}$ & $x_1=e_1$, & $y_1=f_1$, & $y_2=f_2$, & $y_3=f_3$, & $y_4=tf_4$. \\ \hline
$(1|4)_{6}\to(1|4)_{8}$ & $x_1=e_1$, & $y_1=t^{-1}f_1$, & $y_2=t^{-1}(f_2-f_4)$, & $y_3=f_4$, & $y_4=f_3$. \\ \hline
$(1|4)_{7}\to(1|4)_{6}$ & $x_1=e_1$, & $y_1=f_1$, & $y_2=f_2$, & $y_3=f_3$, & $y_4=tf_4$. \\ \hline
\end{longtable}


\begin{proposition}\label{prop:rigid51}\ 
\begin{itemize}
\item The Lie superalgebra $(4|1)_6$ is rigid in the variety $\mathcal{N}_{(4|1)}$.
\item The Lie superalgebra $(1|4)_7$ is rigid in the variety $\mathcal{N}_{(1|4)}$.
\item The Lie superalgebra $(1|4)_4$ is rigid in the variety $\mathcal{LS}_{(1|4)}$.
\end{itemize}
\end{proposition}

\begin{proof}
\begin{align*}
(H^2((4|1)_6,(4|1)_6))_0=&\Span\{ e_1^*\wedge e_2^*\otimes e_2,\ e_1^*\wedge e_3^*\otimes e_2,\ e_2^*\wedge e_3^*\otimes e_1\}.\\
(H^2((1|4)_7,(1|4)_7))_0=&\Span\{ e_1^*\wedge f_1^*\otimes f_1,\ e_1^*\wedge f_1^*\otimes f_2,\ e_1^*\wedge f_1^*\otimes f_3,\ e_1^*f_1^*\otimes f_4\}.\\
(H^2((1|4)_4,(1|4)_4))_0=&\, 0.
\end{align*}
Notice that the deformed Lie superalgebra $\g_t$ of $(4|1)_6$ has product
$$\[e_1,e_2\]_t=e_3+\alpha e_2,\quad \[e_1,e_3\]_t=e_4+\beta e_2,\quad \[e_2,e_3\]_t=\gamma e_1,\quad \[f_1,f_1\]_t=e_4,$$
for some $\alpha,\beta,\gamma\in\C$. If $\alpha\neq0$, $\beta\neq0$, or $\gamma\neq0$, it is easy to see that $\g_t$ is not nilpotent.

\medskip

The same result follows analogously for $(1|4)_7$.

\end{proof}

\begin{theorem}\ 
\begin{enumerate}
\item The irreducible component of the variety $\mathcal{N}_{(5|0)}$ is $\mathcal{C}_1=\overline{O((5|0)_3)}$.
\item The irreducible component of the variety $\mathcal{N}_{(4|1)}$ is $\mathcal{C}_1=\overline{O((4|1)_6)}$.
\item The irreducible components of the variety $\mathcal{N}_{(1|4)}$ are:
\begin{enumerate}
\item $\mathcal{C}_1=\overline{O((1|4)_7)}$.
\item $\mathcal{C}_2=\overline{O((1|4)_4)}$.
\end{enumerate}
\item The irreducible component of the variety $\mathcal{N}_{(0|5)}$ is $\mathcal{C}_1=\overline{O((0|5)_0)}$.
\end{enumerate}
\end{theorem}

\medskip

\scriptsize

\begin{center}

\begin{tikzpicture}[->,>=stealth',shorten >=0.08cm,auto,node distance=1.25cm,
                    thick,main node/.style={rectangle,draw,fill=gray!12,rounded corners=1.5ex,font=\sffamily \bf \bfseries },
                    purple node/.style={rectangle,draw, color=purple,fill=gray!12,rounded corners=1.5ex,font=\sffamily \bf \bfseries },
                    orange node/.style={rectangle,draw, color=orange,fill=gray!12,rounded corners=1.5ex,font=\sffamily \bf \bfseries },
                    green node/.style={rectangle,draw, color=green,fill=gray!12,rounded corners=1.5ex,font=\sffamily \bf \bfseries },
                    olive node/.style={rectangle,draw, color=olive,fill=gray!12,rounded corners=1.5ex,font=\sffamily \bf \bfseries },
                    connecting node/.style={circle, draw, color=purple },
                    rigid node/.style={rectangle,draw,fill=black!20,rounded corners=1.5ex,font=\sffamily \tiny \bfseries },style={draw,font=\sffamily \scriptsize \bfseries }]
                    
\node (181)   {$\dim O(\g)$};

\node (171) [below          of=181]      {$17$};
\node (161) [below          of=171]      {$16$};
\node (151) [below          of=161]      {$15$};
\node (141) [below          of=151]      {$14$};
\node (131) [below          of=141]      {$13$};
\node (121) [below          of=131]      {$12$};
\node (111) [below          of=121]      {$11$};
\node (101) [below          of=111]      {$10$};
\node (91) [below          of=101]      {$9$};
\node (81) [below          of=91]      {$8$};
\node (71) [below          of=81]      {$7$};
\node (61) [below          of=71]      {$6$};
\node (51) [below          of=61]      {$5$};
\node (41) [below          of=51]      {$4$};
\node (31) [below          of=41]      {$3$};
\node (21) [below          of=31]      {$2$};
\node (11) [below          of=21]      {$1$};
\node (01) [below          of=11]      {$0$};

\node (172) [right          of=171]      {};
\node (173) [right          of=172]      {};
\node (162) [right          of=161]      {};
\node (163) [right          of=162]      {};
\node (152) [right          of=151]      {};
\node (153) [right          of=152]      {};
\node (142) [right          of=141]      {};
\node (143) [right          of=142]      {};
\node (132) [right          of=131]      {};
\node (133) [right          of=132]      {};

\node (122) [right          of=121]      {};
\node (123) [right          of=122]      {};
\node (124) [right          of=123]      {};
\node (125) [right          of=124]      {};
\node (126) [right          of=125]      {};
\node (127) [right          of=126]      {};
\node (128) [right          of=127]      {};
\node (129) [right          of=128]      {};

\node (112) [right          of=111]      {};

\node (102) [right          of=101]      {};
\node (103) [right          of=102]      {};
\node (104) [right          of=103]      {};
\node (105) [right          of=104]      {};
\node (106) [right          of=105]      {};
\node (107) [right          of=106]      {};
\node (108) [right          of=107]      {};
\node (109) [right          of=108]      {};

\node (92) [right          of=91]      {};
\node (93) [right          of=92]      {};
\node (94) [right          of=93]      {};
\node (95) [right          of=94]      {};
\node (96) [right          of=95]      {};
\node (97) [right          of=96]      {};
\node (98) [right          of=97]      {};
\node (99) [right          of=98]      {};

\node (82) [right          of=81]      {};
\node (83) [right          of=82]      {};
\node (84) [right          of=83]      {};
\node (85) [right          of=84]      {};
\node (86) [right          of=85]      {};
\node (87) [right          of=86]      {};
\node (88) [right          of=87]      {};
\node (89) [right          of=88]      {};

\node (72) [right          of=71]      {};
\node (73) [right          of=72]      {};
\node (74) [right          of=73]      {};
\node (75) [right          of=74]      {};
\node (76) [right          of=75]      {};
\node (77) [right          of=76]      {};
\node (78) [right          of=77]      {};
\node (79) [right          of=78]      {};

\node (62) [right          of=61]      {};
\node (63) [right          of=62]      {};
\node (64) [right          of=63]      {};
\node (65) [right          of=64]      {};
\node (66) [right          of=65]      {};
\node (67) [right          of=66]      {};
\node (68) [right          of=67]      {};
\node (69) [right          of=68]      {};

\node (52) [right          of=51]      {};

\node (42) [right          of=41]      {};
\node (43) [right          of=42]      {};
\node (44) [right          of=43]      {};
\node (45) [right          of=44]      {};
\node (46) [right          of=45]      {};
\node (47) [right          of=46]      {};
\node (48) [right          of=47]      {};
\node (49) [right          of=48]      {};

\node (32) [right          of=31]      {};
\node (22) [right          of=21]      {};
\node (12) [right          of=11]      {};
\node (02) [right          of=01]      {};
\node (03) [right          of=02]      {};
\node (04) [right          of=03]      {};
\node (05) [right          of=04]      {};
\node (06) [right          of=05]      {};
\node (07) [right          of=06]      {};
\node (08) [right          of=07]      {};
\node (09) [right          of=08]      {};
\node (010) [right          of=09]      {};
\node (011) [right          of=010]      {};

\node [main node] (503)  [right of =172]                      {$(5|0)_3$ };
\node [main node] (504)  [right of =162]                      {$(5|0)_4$ };
\node [main node] (505)  [right of =151]                      {$(5|0)_5$ };
\node [main node] (506)  [right of =153]                      {$(5|0)_6$ };
\node [main node] (502)  [right of =142]                      {$(5|0)_2$ };
\node [main node] (508)  [right of =122]                      {$(5|0)_8$ };
\node [main node] (507)  [right of =102]                      {$(5|0)_7$ };
\node [main node] (501)   [right of =92]                       {$(5|0)_1$ };
\node [main node] (500)   [right of =02]                       {$(5|0)_0$ };

\node [main node] (416)  [right of =105]                      {$(4|1)_6$ };
\node [main node] (415)   [right of =95]                       {$(4|1)_5$ };
\node [main node] (414)   [right of =85]                       {$(4|1)_4$ };
\node [main node] (413)   [right of =75]                       {$(4|1)_3$ };
\node [main node] (412)   [right of =65]                       {$(4|1)_2$ };
\node [main node] (411)   [right of =45]                       {$(4|1)_1$ };
\node [main node] (410)   [right of =05]                       {$(4|1)_0$ };

\node [main node] (147)  [right of =128]                      {$(1|4)_7$ };
\node [main node] (144)  [right of =107]                      {$(1|4)_4$ };
\node [main node] (146)  [right of =109]                      {$(1|4)_6$ };
\node [main node] (143)  [right of =98]                      {$(1|4)_3$ };
\node [main node] (148)  [right of =88]                      {$(1|4)_8$ };
\node [main node] (142)  [right of =78]                      {$(1|4)_2$ };
\node [main node] (145)  [right of =68]                      {$(1|4)_5$ };
\node [main node] (141)  [right of =48]                      {$(1|4)_1$ };
\node [main node] (140)  [right of =08]                      {$(1|4)_0$ };

\node [main node] (050)  [right of =011]                      {$(0|5)_0$ };

\path[every node/.style={font=\sffamily\small}]

(503)  edge [bend right=0, color=black] node{}  (504)
(503)  edge [bend right=20, color=black] node{}  (505)
(503)  edge [bend right=-20, color=black] node{}  (506)
(504)  edge [bend right=0, color=black] node{}  (502)
(505)  edge [bend right=0, color=black] node{}  (502)
(505)  edge [bend right=20, color=black] node{}  (507)
(506)  edge [bend right=0, color=black] node{}  (502)
(502)  edge [bend right=0, color=black] node{}  (508)
(508)  edge [bend right=-40, color=black] node{}  (501)
(507)  edge [bend right=0, color=black] node{}  (501)
(501)  edge [bend right=0, color=black] node{}  (500)

(416)  edge [bend right=0, color=black] node{}  (415)
(416)  edge [bend right=-50, color=black] node{}  (414)
(415)  edge [bend right=50, color=black] node{}  (412)
(414)  edge [bend right=0, color=black] node{}  (413)
(413)  edge [bend right=0, color=black] node{}  (412)
(413)  edge [bend right=-50, color=black] node{}  (411)
(412)  edge [bend right=30, color=black] node{}  (410)
(411)  edge [bend right=0, color=black] node{}  (410)

(147)  edge [bend right=0, color=black] node{}  (146)
(144)  edge [bend right=0, color=black] node{}  (143)
(146)  edge [bend right=-30, color=black] node{}  (148)
(143)  edge [bend right=-60, color=black] node{}  (142)
(148)  edge [bend right=50, color=black] node{}  (145)
(142)  edge [bend right=-50, color=black] node{}  (141)
(145)  edge [bend right=30, color=black] node{}  (140)
(141)  edge [bend right=0, color=black] node{}  (140);

\end{tikzpicture}
\end{center}

\normalsize

\subsection{Case II}

\begin{theorem}
Nilpotent Lie superalgebras of dimension $(3|2)$ are, up to isomorphism:
$$
\begin{array}{lllll}
(3|2)_0: &    \[\cdot, \cdot \]=0.\\ 
(3|2)_1: &  \[ f_1, f_1 \]=e_1,   & \[ f_2, f_2\]= e_2. & & \\
(3|2)_2: &  \[ f_1, f_1 \]=e_1,   & \[ f_2, f_2\]= e_1. & & \\
(3|2)_3: &  \[ f_1, f_1 \]=e_1.   & & & \\
(3|2)_4: &  \[ f_1, f_2 \]=e_1,   & \[ f_2, f_2\]= e_2. & & \\
(3|2)_5: &\[ f_1, f_1\]=e_2,    & \[ f_1, f_2 \]=e_1,   & \[ f_2, f_2\]= e_3.& \\
(3|2)_6: &  \[ e_1, f_2 \]=f_1.   &  & & \\
(3|2)_7: &  \[ e_1, f_2 \]=f_1,   & \[ f_2, f_2\]= e_2. & & \\
(3|2)_{8}: &    \[e_1, e_2 \]=e_3. &&& \\
(3|2)_{9}: &    \[e_1, e_2 \]=e_3, &  \[ f_1, f_1 \]=e_3.   &   & \\
(3|2)_{10}: &    \[e_1, e_2 \]=e_3, &  \[ f_1, f_1 \]=e_3,   & \[ f_2, f_2\]= e_3.  & \\
(3|2)_{11}: &    \[e_1, e_2 \]=e_3, &  \[ e_1, f_2 \]=f_1.   &  & \\
(3|2)_{12}: &    \[e_1, e_2 \]=e_3, &  \[ e_1, f_2 \]=f_1,   & \[f_2,f_2\]=e_3. &  \\
(3|2)_{13}: &    \[e_1, e_2 \]=e_3, &  \[ e_1, f_2 \]=f_1,   & \[f_1,f_2\]=e_3, & \[f_2, f_2\]=2 e_2.    \\
\end{array}
$$

\end{theorem}

\begin{longtable}{|c|c|}
\caption[]{Non-degenerations}
\label{table:non-deg52}\\
\hline
$\g\not\to\h$ & Reason\\
\hline
\endfirsthead
\caption[]{(continued)}\\
\hline
$\g\not\to\h$ & Reason\\
\hline
\endhead
$(3|2)_9\not\to(3|2)_6;\ (3|2)_2\not\to(3|2)_6;\ (3|2)_{10}\not\to(3|2)_6;$ & \multirow{2}{*}{Lemma \ref{lem:invariants} (\ref{inv:der}) $i=1$}\\
$(3|2)_4\not\to(3|2)_7,(3|2)_6;\ (3|2)_{1}\not\to(3|2)_7,(3|2)_6;\ (3|2)_{5}\not\to(3|2)_7,(3|2)_6$ & \\ \hline

$(3|2)_7\not\to(3|2)_2;\ (3|2)_{12}\not\to(3|2)_{10},(3|2)_2$ & Lemma \ref{lem:invariants} (\ref{inv:centro}) $i=1$\\ \hline

$(3|2)_{11}\not\to(3|2)_2,(3|2)_9,(3|2)_3$ & Lemma \ref{lem:invariants} (\ref{inv:gamma})\\ 
\hline
\end{longtable}


\scriptsize

\begin{longtable}{|c|lllll|c|}
\caption[]{Degenerations Case II}
\label{table:deg}\\
\hline
$\g\to\h$ & \multicolumn{5}{|c|}{Parametrized Basis}\\
\hline
\endfirsthead
\caption[]{(continued)}\\
\hline
$\g\to\h$ & \multicolumn{5}{|c|}{Parametrized Basis}\\
\hline
\endhead
$(3|2)_9\to(3|2)_8$ & $x_1=e_1$, & $x_2=e_2$, & $x_3=e_3$, & $y_1=tf_1$, & $y_2=f_2$. \\ \hline
$(3|2)_9\to(3|2)_3$ & $x_1=e_1$, & $x_2=te_2$, & $x_3=e_3$, & $y_1=f_1$, & $y_2=f_2$. \\ \hline
$(3|2)_2\to(3|2)_3$ & $x_1=e_1$, & $x_2=e_2$, & $x_3=e_3$, & $y_1=f_1$, & $y_2=tf_2$. \\ \hline
$(3|2)_7\to(3|2)_3$ & $x_1=e_2$, & $x_2=te_1$, & $x_3=e_3$, & $y_1=f_2$, & $y_2=f_1$. \\ \hline
$(3|2)_7\to(3|2)_6$ & $x_1=e_1$, & $x_2=e_2$, & $x_3=e_3$, & $y_1=tf_1$, & $y_2=tf_2$. \\ \hline
$(3|2)_{10}\to(3|2)_2$ & $x_1=e_3$, & $x_2=te_1$, & $x_3=e_2$, & $y_1=f_1$, & $y_2=f_2$. \\ \hline
$(3|2)_{10}\to(3|2)_9$ & $x_1=e_1$, & $x_2=e_2$, & $x_3=e_3$, & $y_1=f_1$, & $y_2=tf_2$. \\ \hline
$(3|2)_{11}\to(3|2)_6$ & $x_1=e_1$, & $x_2=te_2$, & $x_3=e_3$, & $y_1=f_1$, & $y_2=f_2$. \\ \hline
$(3|2)_{11}\to(3|2)_8$ & $x_1=e_1$, & $x_2=e_2$, & $x_3=e_3$, & $y_1=f_1$, & $y_2=tf_2$. \\ \hline
$(3|2)_4\to(3|2)_2$ & $x_1=e_1$, & $x_2=t^{-1}e_2$, & $x_3=e_3$, & $y_1=-\frac{i}{2}f_1+if_2$, & $y_2=\frac{1}{2}f_1+f_2$. \\ \hline
$(3|2)_{12}\to(3|2)_7$ & $x_1=e_1$, & $x_2=e_3$, & $x_3=te_2$, & $y_1=f_1$, & $y_2=f_2$. \\ \hline
$(3|2)_{12}\to(3|2)_{11}$ & $x_1=e_1$, & $x_2=e_2$, & $x_3=e_3$, & $y_1=tf_1$, & $y_2=tf_2$. \\ \hline
$(3|2)_{12}\to(3|2)_9$ & $x_1=e_1$, & $x_2=e_2$, & $x_3=e_3$, & $y_1=f_2$, & $y_2=t^{-1}f_1$. \\ \hline
$(3|2)_{1}\to(3|2)_4$ & $x_1=\frac{t}{4}(e_1+e_2)$, & $x_2=\frac{1}{4}(e_2-e_1)$, & $x_3=e_3$, & $y_1=\frac{t}{2}(if_1+f_2)$, & $y_2=\frac{1}{2}(f_2-if_1)$. \\ \hline
$(3|2)_{13}\to(3|2)_4$ & $x_1=e_3$, & $x_2=2e_2$, & $x_3=te_1$, & $y_1=f_1$, & $y_2=f_2$. \\ \hline
$(3|2)_{13}\to(3|2)_{12}$ & $x_1=e_1$, & $x_2=e_2-\frac{t^{-1}}{2}e_3$, & $x_3=e_3$, & $y_1=tf_1$, & $y_2=tf_2$. \\ \hline
$(3|2)_{13}\to(3|2)_{10}$ & $x_1=te_1$, & $x_2=t^{-1}e_2$, & $x_3=e_3$, & $y_1=-i(f_1-f_2)$, & $y_2=f_1+f_2$. \\ \hline
$(3|2)_{5}\to(3|2)_{1}$ & $x_1=e_2$, & $x_2=e_3$, & $x_3=t^{-1}e_1$, & $y_1=f_1$, & $y_2=f_2$. \\ \hline

\end{longtable}


\normalsize

\begin{proposition}\label{prop:rigid52}
The Lie superalgebras $(3|2)_5$ and $(3|2)_{13}$ are rigid in the variety $\mathcal{LS}_{(3|2)}$.
\end{proposition}

\begin{proof}
\begin{align*}
(H^2((3|2)_5,(3|2)_5))_0=&\,0,\\
(H^2((3|2)_{13},(3|2)_{13}))_0=&\,0.
\end{align*}
\end{proof}

\medskip

\begin{theorem}
The irreducible components of the variety $\mathcal{N}_{(3|2)}$ are:
\begin{enumerate}
\item $\mathcal{C}_1=\overline{O((3|2)_5)}$.
\item $\mathcal{C}_2=\overline{O((3|2)_{13})}$.
\end{enumerate}
\end{theorem}

\medskip

\begin{center}

\begin{tikzpicture}[->,>=stealth',shorten >=0.08cm,auto,node distance=1.25cm,
                    thick,main node/.style={rectangle,draw,fill=gray!12,rounded corners=1.5ex,font=\sffamily \bf \bfseries },
                    purple node/.style={rectangle,draw, color=purple,fill=gray!12,rounded corners=1.5ex,font=\sffamily \bf \bfseries },
                    orange node/.style={rectangle,draw, color=orange,fill=gray!12,rounded corners=1.5ex,font=\sffamily \bf \bfseries },
                    green node/.style={rectangle,draw, color=green,fill=gray!12,rounded corners=1.5ex,font=\sffamily \bf \bfseries },
                    olive node/.style={rectangle,draw, color=olive,fill=gray!12,rounded corners=1.5ex,font=\sffamily \bf \bfseries },
                    connecting node/.style={circle, draw, color=purple },
                    rigid node/.style={rectangle,draw,fill=black!20,rounded corners=1.5ex,font=\sffamily \tiny \bfseries },style={draw,font=\sffamily \scriptsize \bfseries }]
                    
\node (101)   {$\dim O(\g)$};

\node (91) [below          of=101]      {$9$};
\node (81) [below          of=91]      {$8$};
\node (71) [below          of=81]      {$7$};
\node (61) [below          of=71]      {$6$};
\node (51) [below          of=61]      {$5$};
\node (41) [below          of=51]      {$4$};
\node (31) [below          of=41]      {$3$};
\node (21) [below          of=31]      {$2$};
\node (11) [below          of=21]      {$1$};
\node (01) [below          of=11]      {$0$};

\node (92) [right          of=91]      {};
\node (93) [right          of=92]      {};
\node (94) [right          of=93]      {};
\node (95) [right          of=94]      {};
\node (96) [right          of=95]      {};
\node (97) [right          of=96]      {};
\node (98) [right          of=97]      {};
\node (99) [right          of=98]      {};

\node (82) [right          of=81]      {};
\node (83) [right          of=82]      {};
\node (84) [right          of=83]      {};
\node (85) [right          of=84]      {};
\node (86) [right          of=85]      {};
\node (87) [right          of=86]      {};
\node (88) [right          of=87]      {};
\node (89) [right          of=88]      {};

\node (72) [right          of=71]      {};
\node (73) [right          of=72]      {};
\node (74) [right          of=73]      {};
\node (75) [right          of=74]      {};
\node (76) [right          of=75]      {};
\node (77) [right          of=76]      {};
\node (78) [right          of=77]      {};
\node (79) [right          of=78]      {};

\node (62) [right          of=61]      {};
\node (63) [right          of=62]      {};
\node (64) [right          of=63]      {};
\node (65) [right          of=64]      {};
\node (66) [right          of=65]      {};
\node (67) [right          of=66]      {};
\node (68) [right          of=67]      {};
\node (69) [right          of=68]      {};

\node (52) [right          of=51]      {};
\node (53) [right          of=52]      {};
\node (54) [right          of=53]      {};

\node (42) [right          of=41]      {};
\node (43) [right          of=42]      {};
\node (44) [right          of=43]      {};
\node (45) [right          of=44]      {};
\node (46) [right          of=45]      {};
\node (47) [right          of=46]      {};
\node (48) [right          of=47]      {};
\node (49) [right          of=48]      {};

\node (32) [right          of=31]      {};
\node (33) [right          of=32]      {};
\node (34) [right          of=33]      {};

\node (22) [right          of=21]      {};
\node (12) [right          of=11]      {};
\node (02) [right          of=01]      {};
\node (03) [right          of=02]      {};
\node (04) [right          of=03]      {};
\node (05) [right          of=04]      {};
\node (06) [right          of=05]      {};
\node (07) [right          of=06]      {};
\node (08) [right          of=07]      {};
\node (09) [right          of=08]      {};
\node (010) [right          of=09]      {};
\node (011) [right          of=010]      {};

\node [main node] (325)  [right of =93]                      {$(3|2)_5$ };
\node [main node] (321)   [right of =82]                       {$(3|2)_1$ };
\node [main node] (3213)   [right of =84]                       {$(3|2)_{13}$ };
\node [main node] (324)   [right of =72]                       {$(3|2)_4$ };
\node [main node] (3212)   [right of =74]                       {$(3|2)_{12}$ };
\node [main node] (327)   [right of =61]                       {$(3|2)_7$ };
\node [main node] (3210)   [right of =63]                       {$(3|2)_{10}$ };
\node [main node] (3211)  [right of =65]                      {$(3|2)_{11}$ };
\node [main node] (322)   [right of =52]                       {$(3|2)_2$ };
\node [main node] (329)   [right of =54]                       {$(3|2)_9$ };
\node [main node] (323)   [right of =42]                       {$(3|2)_3$ };
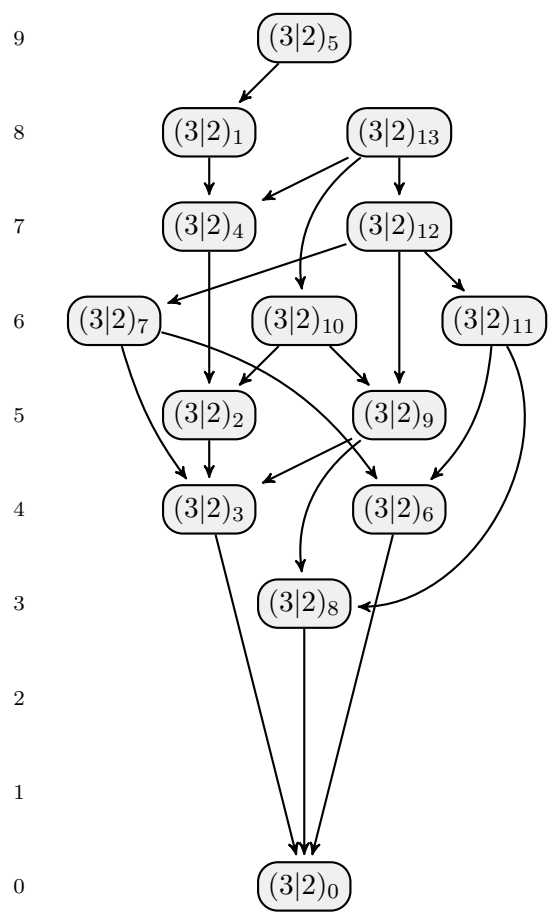
\node [main node] (326)   [right of =44]                       {$(3|2)_6$ };
\node [main node] (328)   [right of =33]                       {$(3|2)_8$ };
\node [main node] (320)   [right of =03]                       {$(3|2)_0$ };

\path[every node/.style={font=\sffamily\small}]

(325)  edge [bend right=0, color=black] node{}  (321)
(321)  edge [bend right=0, color=black] node{}  (324)
(3213)  edge [bend right=0, color=black] node{}  (324)
(3213)  edge [bend right=0, color=black] node{}  (3212)
(3213)  edge [bend right=30, color=black] node{}  (3210)
(324)  edge [bend right=0, color=black] node{}  (322)
(3212)  edge [bend right=0, color=black] node{}  (327)
(3212)  edge [bend right=0, color=black] node{}  (329)
(3212)  edge [bend right=0, color=black] node{}  (3211)
(327)  edge [bend right=10, color=black] node{}  (323)
(327)  edge [bend right=-20, color=black] node{}  (326)
(3210)  edge [bend right=0, color=black] node{}  (322)
(3210)  edge [bend right=0, color=black] node{}  (329)
(3211)  edge [bend right=-60, color=black] node{}  (328)
(3211)  edge [bend right=-20, color=black] node{}  (326)
(322)  edge [bend right=0, color=black] node{}  (323)
(329)  edge [bend right=0, color=black] node{}  (323)
(329)  edge [bend right=30, color=black] node{}  (328)
(323)  edge [bend right=0, color=black] node{}  (320)
(326)  edge [bend right=0, color=black] node{}  (320)
(328)  edge [bend right=0, color=black] node{}  (320);

\end{tikzpicture}
\end{center}

\subsection{Case III}\ 

\medskip

In this section we provide the techniques, mentioned in Section 2, that we  use to obtain the
algebraic classification for the particular (and most difficult) case of Lie superalgebras $([\cdot, \cdot], \rho,  \Gamma)$ of  dimension $(2|3)$ such that $[\cdot, \cdot]=0 $ and $\rho=0$.

\medskip

As we mentioned previously,  classifying Lie superalgebras under the above hypothesis  is equivalent to finding the orbits of the action of $\GL_2(\C) \oplus \GL_3(\C)$ on $\Sym_3(\C) \times \Sym_3(\C)$ given by:
\begin{equation}
(T,S) \cdot (\Gamma^1, \Gamma^2)= 
(T_{11}  S^t \Gamma^1S+ T_{12}  S^t \Gamma^2S, \;  T_{21}  S^t \Gamma^1S+ T_{22}  S^t \Gamma^2S).
\label{action23}
\end{equation} 
Hence, in order to find the orbit  
of  $(\Gamma^1, \Gamma^2)$, which we  denote by $[ \Gamma^1, \Gamma^2]$, it is useful to ask if there exists $S\in\GL_3(\C)$ for which  $S^t\Gamma^1S$ and  $S^t\Gamma^2S $ are diagonal matrices, in which case we obtain the following result.

\begin{proposition}
Let $(\Gamma^1, \Gamma^2)$ be such that $\{\Gamma^1, \Gamma^2\}$ is simultaneously dia\-go\-na\-li\-za\-ble. Then,  $(\Gamma^1, \Gamma^2)$ belongs to one and only one of the following orbits:

\begin{centering}
{\setstretch{1.8}
\begin{tabular}{llcll}
$(2|3)_0$: &   $[ 0, 0 ]$,   &$\quad$&   $(2|3)_4$: & $[\I_1, \I_2 ]$,
\\ 
$(2|3)_1$: &   $[ \I_1, 0]$,&$\quad$&    $(2|3)_5$: & $[\I_1+I_3, \I_2 ]$,
\\ 
$(2|3)_2$: & $[\I_1+ \I_2, 0]$,&$\quad$&    $(2|3)_6$: & $[\I_1+I_3, \I_2+\I_3 ]$,\\
$(2|3)_3$: &$[\id_3, 0]$,&&&\\
\end{tabular}
}
\end{centering}

where 
$$I_1= 
\begin{psmallmatrix}
1 &0&0\\
0 &0&0\\
 0&0&0
\end{psmallmatrix}, \quad I_2= 
\begin{psmallmatrix}
0 &0&0\\
0 &1&0\\
 0&0&0
\end{psmallmatrix}, \quad \text{and}\quad I_3= 
\begin{psmallmatrix}
0 &0&0\\
0 &0&0\\
 0&0&1
\end{psmallmatrix}. 
$$
(see \cite{Her}).
\end{proposition}

In order to find the orbit of a pair $(\Gamma^1, \Gamma^2)$ such that $\{\Gamma^1, \Gamma^2\}$  is non-simultaneously diagonalizable, we observe the following. 
\begin{lemma}\ 
\begin{enumerate}[(i)]
\item If $\rk(\Gamma^i) \leq 1$ for $i\in \{1,2\}$, then $\{ \Gamma^1, \Gamma^2\}$ is simultaneously diagonali\-zable.
\item If $\{ \Gamma^1, \Gamma^2 \}$ is simultaneously diagonalizable,  then  for all $( {\Gamma^\prime}^1, {\Gamma^\prime}^2)$ $\in   
[ \Gamma^1, \Gamma^2 ]$ it follows that  $\{{ \Gamma^\prime}^1, {\Gamma^\prime}^2 \}$
is simultaneously diagonalizable.
\end{enumerate}
\label{rank1sd}
 \end{lemma}
 \begin{proof}
In fact, suppose  that $\rk(\Gamma^i)=1$ for $i\in \{ 1,2\}$. Note that there is $S \in \GL_3(\C)$ such that $S^t \Gamma^1S=\I_1$. Thus, we can assume that $\Gamma^1=\I_1$. Now write 
$\Gamma^2 = \begin{psmallmatrix} 
\alpha & w^t\\
w & A
\end{psmallmatrix}, \text{ where }  a \in \C, \; w\in \C^2, \text{ and } A\in\Sym_2(\C). 
$
Since all minors of $\Gamma^2$ are zero, there is $R\in \GL_2(\C)$ such that 
$R^tAR=
\begin{psmallmatrix}
\epsilon & 0 \\
0 & 0
\end{psmallmatrix}
 $, where $\epsilon \in \{0, 1\}$. Hence, taking  
 $S= \begin{psmallmatrix}  
 1 & 0 \\
 R^tw & R \end{psmallmatrix}
 $, it follows that  $S^t \Gamma^1 S=\I_1$ and  $S^t \Gamma^2 S=\I_2$, i.e. $\{ \Gamma^1, \Gamma^2\}$ is simultaneously diagonalizable.  
 
 For (ii), toward a contradiction, suppose that  $\{ \Gamma^1, \Gamma^2\}$ is non-simultaneously diago\-nalizable but $ \{ { \Gamma^\prime}^1, {\Gamma^\prime}^2 \}$ is simultaneously diagonalizable,
 for some  $({ \Gamma^\prime}^1, {\Gamma^\prime}^2  ) \in [ \Gamma^1, \Gamma^2]$. 
 Let   $(T,S)\in \GL_2(\C) \times \GL_3(\C)$ be such that  $({ \Gamma^\prime}^1, {\Gamma^\prime}^2  ) = (T,S)\cdot (  \Gamma^1, \Gamma^2)$. Thus, there are $i,j \in \{1,2,3 \}$ with $i<j$ such that 
$$
\begin{pmatrix}
(S^t\Gamma^1S)_{ij}\\
(S^t\Gamma^2S)_{ij}
\end{pmatrix} \neq 
\begin{pmatrix}
0 \\
0
\end{pmatrix}
,\text{ and } \; 
T
\begin{pmatrix}
(S^t\Gamma^1S)_{ij}\\
(S^t\Gamma^2S)_{ij}
\end{pmatrix}
=
\begin{pmatrix}
0 \\
0
\end{pmatrix}. 
$$
\end{proof}

Recall that complex symmetric matrices may be non-dia\-go\-na\-li\-za\-ble by orthogo\-nal ones; for such matrices it is convenient to use the Symmetric Normal Form given in \cite{Cr}. This allows us to write a non-diagonalizable matrix  in a convenient form,  as the following result shows  for $n\in \{2,3\}$
(see   \cite{Cr}, pag. 348). 

\begin{lemma}[\cite{Cr}]
Let $n\in \{ 2, 3\}$. If  $A\in \Sym_n(\C)$ is non-diagonalizable,  then there exists $S\in \O_n(\C)$ such that the following holds:
\begin{enumerate}[(i)]
\item
For $n=2$;   $\;S^tAS=
\begin{psmallmatrix}
 \lambda+1 & i \\
 i & \lambda -1 
\end{psmallmatrix}
$,  where $\lambda$ is the eigenvalue of $A$.   
\item 
For $n=3$;  $\; S^tAS=
\begin{psmallmatrix}
 \lambda+1 & i & c \\
 i & \lambda -1 & ic\\
 c & ic& \mu 
\end{psmallmatrix}
$,  where $\lambda$ and $\mu$ are the eigenvalues of $A$, and if $\lambda\neq \mu$, then $c=0$. 
\end{enumerate}
\label{cfnon-diagonalizable}
\end{lemma}

\medskip

\begin{proposition}Let $(\Gamma^1, \Gamma^2)$ be such that $\{ \Gamma^1, \Gamma^2 \}$ is non-simultaneously diago\-nalizable. Then, $(\Gamma^1, \Gamma^2)$  belongs in one and only one of the following orbits:

\begin{centering}
{\setstretch{1.8}
\begin{tabular}{llcll}
$(2|3)_7$: &
 $\left [
 \begin{psmallmatrix}
 K & 0 \\
 0 & 0 
\end{psmallmatrix}, 
 \begin{psmallmatrix}
 L & 0 \\
 0 & 0 
\end{psmallmatrix}  
\right ]$, & $\quad $&
$(2|3)_{9}$: &
 $\left [
 \begin{psmallmatrix}
 K & 0 \\
 0 & 1 
\end{psmallmatrix}, 
 \begin{psmallmatrix}
 L & 0 \\
 0 & 0 
\end{psmallmatrix}  
\right ]$,\\
 $(2|3)_{8}$:&
 $\left [
 \begin{psmallmatrix}
 K & 0 \\
 0 & 0 
\end{psmallmatrix},
 \begin{psmallmatrix}
 L & u_0 \\
 u_0^t & 0 
\end{psmallmatrix}  
\right ]$,&&$(2|3)_{10}$:&
 $\left [
 \begin{psmallmatrix}
 K & 0 \\
 0 & 1 
\end{psmallmatrix}, 
 \begin{psmallmatrix}
 L & 0 \\
 0 & 1 
\end{psmallmatrix}  
\right ]$,
\\
&&& $(2|3)_{11}$:&
 $\left [
 \begin{psmallmatrix}
 K & 0 \\
 0 & 1 
\end{psmallmatrix}
 \begin{psmallmatrix}
 L & u_0 \\
 u_0^t & 0
\end{psmallmatrix}  
\right ]$,
\end{tabular}
}
\end{centering}

where: 
$$K=\begin{psmallmatrix}
0 & 1 \\
1& 0
\end{psmallmatrix}, \quad L=\begin{psmallmatrix}
0 & 0 \\
0& 2
\end{psmallmatrix}, \quad  u_0 = \begin{psmallmatrix}
 0 \\
1
\end{psmallmatrix}.$$

\end{proposition}

\medskip

To prove the previous result,  we observe that from Lemma \ref{rank1sd}, it follows that $\rk(\Gamma^i)\geq2$, for some $i\in \{1,2\}$.  Furthermore:
\begin{enumerate}[(i)]
\item  If $ \Span\{ \Gamma^1, \Gamma^2\} \cap \GL_3(\C)\neq \emptyset$, then   $(\Gamma^1, \Gamma^2) \in [\id_3, \Gamma^3]$, for some $\Gamma^3\in \Sym_3(\C)$ non-diagonalizable. 
\item If $\Span \{ \Gamma^1, \Gamma^2 \} \cap \GL_3(\C)=\emptyset$, then $(\Gamma^1, \Gamma^2) \in [\I_1+\I_2, \Gamma^3]$, for some $\Gamma^3\in \Sym_3(\C)$ such that 
$\Span\{\I_1+ \I_2, \Gamma^3\}\cap \GL_3(\C)=\emptyset$ and $\{\I_1+ \I_2, \Gamma^3\}$ is non-simultaneously diagonalizable. 
\end{enumerate}

In the following lemmas concerning the previous cases, we use the notation: 
$
 \Delta(\lambda)= \begin{psmallmatrix} 
\lambda+1  & i\\
i & \lambda -1
\end{psmallmatrix}$, $T_\lambda= \begin{psmallmatrix} 
-1& 0\\
-\lambda & 1
\end{psmallmatrix}$,   with $\lambda\in\C$,   
$ \,u_1 = \begin{psmallmatrix}
 1 \\
i
\end{psmallmatrix}, \;\; 
R= (\sqrt{2}\;)^{-1}  \begin{psmallmatrix} 
1& -1\\
i & i
\end{psmallmatrix}$ and $ 
 S_0= \begin{psmallmatrix} 
R & 0\\
0 & i
\end{psmallmatrix}.
$

\medskip

\begin{lemma}
 Let $\Gamma^3\in \Sym_3(\C)$. If $\Gamma^3$ is non-diagonalizable, then $(\id_3, \Gamma^3)$ belongs to one and only one of the following orbits:
$$ \left [
 \begin{psmallmatrix}
 K & 0 \\
 0 & 1 
\end{psmallmatrix}, 
 \begin{psmallmatrix}
 L & 0 \\
0 & 0
\end{psmallmatrix}  
\right ], \quad 
 \left [
 \begin{psmallmatrix}
 K & 0 \\
 0 & 1 
\end{psmallmatrix}, 
 \begin{psmallmatrix}
 L & 0 \\
 0 & 1
\end{psmallmatrix}  
\right ], \quad 
 \left [
 \begin{psmallmatrix}
 K & 0 \\
 0 & 1 
\end{psmallmatrix}, 
 \begin{psmallmatrix}
 L & u_0 \\
 u_0^t & 0
\end{psmallmatrix}  
\right ].
$$
\end{lemma}
\begin{proof}
From Lemma \ref{cfnon-diagonalizable} there exists $S\in \O_3(\C)$ such that 
$
S^t \Gamma^3 S = 
\begin{psmallmatrix}
\Delta(\lambda) & cu_1 \\
cu_1^t & \mu 
\end{psmallmatrix}.$
for some $\lambda, \mu, c\in \C$. Then, 
$$
(T_\lambda,  SS_0 )\cdot  (\id_3, \, \Gamma^3)
= 
\left(
\begin{psmallmatrix}
K & 0\\
0 & 1
\end{psmallmatrix},
\; 
\begin{psmallmatrix}
L  &   -i \sqrt{2} c u_1 \\
-i\sqrt{2}c u_1^t& \lambda- \mu
\end{psmallmatrix}   
\right).    
$$
Finally, 
taking 
$$
T^\prime=
\begin{psmallmatrix}
\gamma^2 & 0 \\
0 & 1
\end{psmallmatrix} \text{ and } 
S^\prime= 
\begin{psmallmatrix}
\frac{1}{\gamma^2} & 0 & 0\\
0& 1 &0 \\
0 & 0& \frac{1}{\gamma}
\end{psmallmatrix}, \text{ where }
\gamma=\begin{cases}
\sqrt{\lambda-\mu} & \text{ if } \lambda-\mu \neq0,\\
-i\sqrt{2}c & \text{ if } c\neq 0,
\end{cases} 
$$
it follows that 
$$
(T^\prime, S^\prime)\cdot  
\left(
\begin{psmallmatrix}
K & 0\\
0& 1
\end{psmallmatrix},
\; 
\begin{psmallmatrix}
L&   -i\sqrt{2}c u_1\\
-i\sqrt{2}c u^t_1  & \lambda-\mu
\end{psmallmatrix}
\right)
= 
\begin{cases}
\left(
\begin{psmallmatrix}
K & 0\\
0 & 1
\end{psmallmatrix},
\; 
\begin{psmallmatrix}
L & 0\\
0 & 1
\end{psmallmatrix}
\right),
\\
\left(
\begin{psmallmatrix}
K & 0\\
0 & 1
\end{psmallmatrix},
\; 
\begin{psmallmatrix}
L & u_0\\
u_0& 1
\end{psmallmatrix}
\right).
\end{cases}
$$
respectively. 
\end{proof}
\begin{lemma}
 Let $\Gamma^3\in \Sym_3(\C)$. If  $\Span\{\I_1+ \I_2, \Gamma^3\}\cap \GL_3(\C)=\emptyset$ and $\{\I_1+ \I_2, \Gamma^3\}$ is non-simultaneously diagonalizable, then  
 $(\I_1+\I_2, \Gamma^3)$ belongs to one and only one of the following orbits:
$$
 \left [
 \begin{psmallmatrix}
 K & 0 \\
 0 & 0
\end{psmallmatrix}, 
 \begin{psmallmatrix}
 L & 0 \\
 0 & 0
\end{psmallmatrix}  
\right ], \quad 
 \left [
 \begin{psmallmatrix}
 K & 0 \\
 0 & 0 
\end{psmallmatrix}, 
 \begin{psmallmatrix}
 L & u_0 \\
 u_0^t & 0
\end{psmallmatrix}  
\right ].
$$
\end{lemma}
\begin{proof}
Write
$\Gamma^3= \begin{psmallmatrix}
B & w \\
w^t & b
\end{psmallmatrix}$, where  $B \in \Sym_2(\C)$, $w \in \C^2$,  and $b \in \C$. 
First of all, notice that if $bw\neq 0$, then 
$$
\left (
\begin{psmallmatrix}
\id_2 & 0 \\
0 & 0 
\end{psmallmatrix}, 
\begin{psmallmatrix}
B & w \\
w^t & b
\end{psmallmatrix}
\right ) \in 
\left [
\begin{psmallmatrix}
\id_2 & 0 \\
0 & 0 
\end{psmallmatrix}, 
\begin{psmallmatrix}
B^\prime & 0 \\
0 & b
\end{psmallmatrix}
\right ]
$$
for $B^\prime = B - \frac{1}{b}ww^t  \in \Sym_2(\C)$. In fact, acting on $(\I_1+\I_2, \Gamma^3)$ with $(\id_2, S)$, where  $S=\begin{psmallmatrix}
\id_2 & 0 \\
-\frac{1}{b}  w^t & 1 
\end{psmallmatrix}
$, 
the result is obtained. Hence, we may assume that $bw=0$. 

First take an appropriate $A\in \GL_2(\C)$ for putting  $\{\id_2, B\}$ in canonical form simultaneously. Thus, taking $S=
\begin{psmallmatrix}
A & 0 \\
0 & \alpha
\end{psmallmatrix}
$, for some $\alpha  \in \C \setminus \{ 0\}$, it follows that 
 
\begin{equation}
(\id_2, S)\cdot
 \left ( 
\begin{psmallmatrix}
\id_2 & 0 \\
0 & 0 
\end{psmallmatrix},  \begin{psmallmatrix}
B & w \\
w^t & b
\end{psmallmatrix}
\right )
= 
\left (
\begin{psmallmatrix}
A^tA & 0\\
0 &  0 
\end{psmallmatrix}, 
\begin{psmallmatrix}
A^tBA & \alpha A^t w\\
\alpha w^tA & \alpha^2 b
\end{psmallmatrix}
\right )
\label{blockaction}.
\end{equation}

{\bf Claim.} If $B$ is diagonalizable,  then 
$$ 
\left (
\begin{psmallmatrix}
\id_2 & 0 \\
0 & 0 
\end{psmallmatrix},  \begin{psmallmatrix}
B & w \\
w^t & b
\end{psmallmatrix}
\right ) \in 
\left [
\begin{psmallmatrix}
K & 0 \\
0 & 0 
\end{psmallmatrix},  \begin{psmallmatrix}
L &  u_0 \\
 u_0^t & 0 
\end{psmallmatrix}
\right ]. 
$$ 
In fact, let $A_0\in \O_2(\C)$ such that $A_0^tBA_0=\diag(\lambda, \mu)$. Notice that, since $\{I_1+\I_2, \Gamma^3 \}$ is non-simultaneously diagonalizable, then $w\neq 0$. Hence,  taking $S$ as in \eqref{blockaction} with $A=A_0$ and $\alpha=1$, it follows that:
\begin{equation}
(\id_2, S)\cdot 
\left ( 
\begin{pmatrix}
\id_2 & 0 \\
0 & 0 
\end{pmatrix},  \begin{pmatrix}
B & w \\
w^t & 0
\end{pmatrix}
\right )= 
\left ( 
\begin{pmatrix}
\id_2 & 0 \\
0 & 0 
\end{pmatrix},  \begin{pmatrix}
\diag(\lambda,\mu)&  A_0^tw \\
 w^tA_0 & 0 
\end{pmatrix}
\right ).
\label{blockaction1}
\end{equation}
A straightforward computation shows that: 
$$\Span \left   \{
\begin{psmallmatrix}
\id_2 & 0 \\
0 & 0 \\ 
\end{psmallmatrix}, 
\begin{psmallmatrix}
\diag(\lambda, \mu)& A_0^t w \\
w^t  A_0  & 0  
\end{psmallmatrix} 
\right \} \cap \GL_3(\C) = \emptyset \;
\text{ iff } \lambda=\mu \text{ and } w^tw =0.$$ 
Thus, $B=\lambda \id_2$ and $w^tw=0$. Hence,  there exists  $A_1\in \O_2(\C)$  such that $A_1w=u_1$, (see \cite{Cr}).  Therefore, taking $A_1$ instead of $A_0$ in  \eqref{blockaction1} it follows that 
$$ 
\left (
\begin{psmallmatrix}
\id_2 & 0 \\
0 & 0 
\end{psmallmatrix},  \begin{psmallmatrix}
B & w \\
w^t & b
\end{psmallmatrix}
\right ) \in 
\left [
\begin{psmallmatrix}
\id_2 & 0 \\
0 & 0 
\end{psmallmatrix},  \begin{psmallmatrix}
0 &  u_1 \\
 u_1^t & 0 
\end{psmallmatrix}
\right ] =
\left [
\begin{psmallmatrix}
K & 0 \\
0 & 0 
\end{psmallmatrix},  \begin{psmallmatrix}
L &  u_0 \\
 u_0^t & 0 
\end{psmallmatrix}
\right ]. 
$$ 
The last equality is obtained by acting  with
$\left (
\begin{psmallmatrix}
0 & 2 \\
1 & 0
\end{psmallmatrix}, \;  \frac{1}{2}
\begin{psmallmatrix}
0 & 2 & 1  \\
0 & 0 & i \\
2 & 0 & 0
\end{psmallmatrix}  
\right )  $ on $\left ( \begin{psmallmatrix}
\id_2 & 0 \\
0 & 0 
\end{psmallmatrix},  \begin{psmallmatrix}
0 &  u_1 \\
 u_1^t & 0 
\end{psmallmatrix}
\right )$.

Now assume that $B$ is non diagonalizable. Then, there is $A_0\in \O_2(\C)$ such that $A_0^tBA_0= \Delta ({\lambda} )$, for some $\lambda \in \C$. Take 
 $S=\begin{psmallmatrix}
 A_0R & 0 \\
 0 & \alpha
 \end{psmallmatrix} \in \GL_3(\C)
 $, then
\begin{equation}
(T_\lambda,  S)\cdot 
\left ( 
\begin{psmallmatrix}
\id_2 & 0 \\
0 & 0 
\end{psmallmatrix},  \begin{psmallmatrix}
B & w \\
w^t & b
\end{psmallmatrix}
\right )= 
\left ( 
\begin{psmallmatrix}
K& 0 \\
0 & 0 
\end{psmallmatrix},  \begin{psmallmatrix}
L  &  \alpha (A_0R)^tw \\
\alpha w^tA_0R & \alpha^2b 
\end{psmallmatrix}
\right ).
\label{blockaction2}
\end{equation}
A straightforward computation shows:
\begin{enumerate}[(i)]
\item For $w=0$:
$$\Span \left   \{
\begin{psmallmatrix}
K& 0 \\
0 & 0 \\ 
\end{psmallmatrix}, 
\begin{psmallmatrix}
L &  0 \\
0 & \alpha^2b
\end{psmallmatrix} 
\right \} \cap \GL_3(\C) = \emptyset \;
\text{ iff } b=0.$$ 
\item  For $w\neq 0$:
$$\Span \left   \{
\begin{psmallmatrix}
K& 0 \\
0 & 0 \\ 
\end{psmallmatrix}, 
\begin{psmallmatrix}
L & \alpha (A_0R)^tw  \\
\alpha w^t A_0R & 0
\end{psmallmatrix} 
\right \} \cap \GL_3(\C) = \emptyset \;
\text{ iff }  (A_0R)^tw= \beta u_0, $$
 for some $\beta \in  \C\setminus \{ 0\}$.
\end{enumerate}
Therefore, 
$$ 
\left (
\begin{pmatrix}
\id_2 & 0 \\
0 & 0 
\end{pmatrix},  \begin{pmatrix}
B & 0 \\
0 & 0
\end{pmatrix}
\right ) \in 
\left [
\begin{pmatrix}
K& 0 \\
0 & 0 
\end{pmatrix},  \begin{pmatrix}
L &  0 \\
0 & 0 
\end{pmatrix}
\right ]
$$
 and 
$$ 
\left (
\begin{pmatrix}
\id_2 & 0 \\
0 & 0 
\end{pmatrix},  \begin{pmatrix}
B & w \\
w^t & 0
\end{pmatrix}
\right ) \in 
\left [
\begin{pmatrix}
K& 0 \\
0 & 0 
\end{pmatrix},  \begin{pmatrix}
L &  u_0 \\
 u_0 & 0 
\end{pmatrix}
\right ]. 
$$ 
The last asertion is obtained taking $\alpha= \beta^{-1}$. 
\end{proof}

With all this, we can write the algebraic classification of nilpotent Lie superalgebras of dimension $(2|3)$.

\medskip

\begin{theorem}
Nilpotent Lie superalgebras of dimension $(2|3)$ are, up to isomorphism:
$$
\begin{array}{lllll}
(2|3)_0: &  \[\cdot, \cdot\]=0.&&&\\
(2|3)_1: &  \[ f_1, f_1\]= e_1. &&&\\
(2|3)_2: &  \[ f_1, f_1\]= e_1, & \[f_2, f_2\]=e_1. &&\\
(2|3)_3: &  \[ f_1, f_1\]= e_1, & \[f_2, f_2\]=e_1,& \[f_3, f_3\]=e_1. &\\ 
(2|3)_4: &  \[ f_1, f_1\]= e_1, & \[f_2, f_2\]=e_2.&&\\
(2|3)_5: & \[ f_1, f_1\]= e_1, & \[f_2, f_2\]=e_2, & \[f_3, f_3\]=e_1. &    \\
(2|3)_6: &  \[ f_1, f_1\]= e_1, & \[f_2, f_2\]=e_2, & \[f_3, f_3\]=e_1+e_2. &    \\
(2|3)_7: &  \[ f_1, f_2\]= e_1, & \[f_2, f_2\]=2e_2.&  &\\
(2|3)_8: & \[ f_1, f_2\]= e_1,& \[f_2, f_2\]=2e_2, & \[f_2, f_3\]= e_2. &  \\
(2|3)_{9}: & \[ f_1, f_2\]= e_1,& \[f_2, f_2\]=2e_2, &\[f_3, f_3\]= e_1. &  \\
(2|3)_{10}: &  \[ f_1, f_2\]= e_1,& \[f_2, f_2\]=2e_2, & \[f_3, f_3\]= e_1+e_2. &  \\
(2|3)_{11}: & \[ f_1, f_2\]= e_1, & \[f_2, f_2\]=2e_2, & \[f_2, f_3\] = e_2, & 
\[f_3, f_3\]=  e_1.  \\
(2|3)_{12}: &\[e_1, f_3 \]= f_1.& & & \\
(2|3)_{13}: &\[e_1, f_3 \]= f_1,&\[ f_2, f_2 \]= e_2. & &  \\
(2|3)_{14}: &\[e_1, f_3 \]= f_1,& \[ f_2, f_3 \]= e_2. & & \\
(2|3)_{15}: &\[e_1, f_3 \]= f_1,&\[ f_3, f_3 \]= e_2.& & \\
(2|3)_{16}: &\[e_1, f_3 \]= f_1,&\[ f_2, f_2 \]= e_2,&\[ f_3, f_3  \]= e_2. &\\
(2|3)_{17}: &\[e_1, f_3 \]= f_1,&\[e_2, f_2 \]= f_1. &  &\\
(2|3)_{18}:&\[e_1, f_3 \]= f_1, &\[e_2, f_2 \]= f_1, & \[ f_2, f_2 \] = 2 e_1, &\[f_2, f_3\]= -e_2.   \\
(2|3)_{19}:&\[e_1, f_3 \]= f_1, &\[e_2, f_2 \]= f_1, & \[ f_2, f_3 \] = -  e_1, & \[f_3, f_3\]= 2e_2.  \\
(2|3)_{20}:& \[e_1, f_3 \]= f_1, & \[e_2, f_3 \]= f_2. & \\
(2|3)_{21}:& \[e_1, f_2\]= f_1, &\[e_1, f_3\]= f_2. &                         &  \\ 
(2|3)_{22}:&\[e_1, f_2\]= f_1, & \[e_1, f_3\]= f_2, &\[f_3, f_3\]=e_2. & \\ 
(2|3)_{23}:&\[e_1, f_2\]= f_1, & \[e_1, f_3\]= f_2, & \[f_1, f_3 \]=-e_2, & \[f_2, f_2 \]=e_2.  \\ 
(2|3)_{24}:&\[e_1, f_2\]= f_1,& \[e_1, f_3\]= f_2, & \[e_2, f_3\]= f_1.    &\\
\end{array}
$$
\end{theorem}

\medskip

\begin{remark}
Lie superalgebras $(2|3)_{14}$, $(2|3)_{18}$, $(2|3)_{19}$, $(2|3)_{22}$, $(2|3)_{23}$ are missing in the classification given in \cite{He}. On the other hand,  Matiadou and Fe\-llouris provided in \cite{MF} a classification of 5-dimensional Lie superalgebras over $\C$. Their classification includes, for dimension $(2|3)$, three parametric families which are in fact not infinite families but rather four isomorphism classes of Lie superalgebras. 
The isomorphisms between our classification and that of Matiadou and Fellouris are the following:

\begin{enumerate}[(i)]
\item $(2A_{1,1}+3A)^5 \simeq (2|3)_{10}$,  for $\lambda \in \C \setminus \{ 0,  i \}$.
\item $(2A_{1,1}+3A)^5  \simeq (2|3)_{11}$,  for $\lambda=i$.
\item $(2A_{1,1}+3A)^6  \simeq (2|3)_{6}$,  for $\kappa, \lambda   \in \C \setminus\{0\}$  such that $(1-\kappa^2-\lambda^2)^2-4\lambda^2\neq 0$.
\item $(2A_{1,1}+3A)^6  \simeq (2|3)_{9}$,  for  $\kappa, \lambda    \in \C \setminus\{0\}$ such that $(1-\kappa^2-\lambda^2)^2-4\lambda^2 =0$.
\item $(2A_{1,1}+3A)^7  \simeq (2|3)_{6}$,  for $\lambda  \in \C \setminus\{0,  \pm\frac{ i}{2}\}$.
\item $(2A_{1,1}+3A)^7  \simeq (2|3)_{9}$, for $\lambda = \pm\frac{i}{2}$. 
\end{enumerate} 
\label{cotejo2-3}
Moreover, Lie superalgebras $(2|3)_{14}$, $(2|3)_{18}$ and $(2|3)_{19}$ are also missing.
\end{remark}

\begin{proposition}\label{prop:rig_23}
The following Lie superalgebras are rigid in the variety $\mathcal{N}_{(2|3)}$: $(2|3)_6$, $(2|3)_{18}$, $(2|3)_{19}$, $(2|3)_{23}$ and $(2|3)_{24}$.
\end{proposition}

\begin{proof}
\begin{align*}
(H^2((2|3)_6,(2|3)_6))_0=&\ 0.\\
(H^2((2|3)_{18},(2|3)_{18}))_0=&\Span\{ e_1^*\wedge e_2^*\otimes e_2-e_1^*\wedge f_3^*\otimes f_3+e_2^*\wedge f_2^*\otimes f_3,\\
&-e_1^*\wedge e_2^*\otimes e_1+e_2^*\wedge f_1^*\otimes f_1+f_1^*\wedge f_2^*\otimes e_1\}.\\
(H^2((2|3)_{19},(2|3)_{19}))_0=&\Span\{ e_1^*\wedge e_2^*\otimes e_2+e_1^*\wedge f_1^*\otimes f_1+f_1^*\wedge f_3^*\otimes e_2,\\
&-e_1^*\wedge e_2^*\otimes e_1+e_1^*\wedge f_3^*\otimes f_2+e_2^*\wedge f_2^*\otimes f_2\}.\\
(H^2((2|3)_{23},(2|3)_{23}))_0=&\Span\{ e_1^*\wedge f_1^*\otimes f_1-e^*_1\wedge f_3^*\otimes f_3,\ e_1^*\wedge f_1^*\otimes f_2+e^*_1\wedge f_2^*\otimes f_3\}.\\
(H^2((2|3)_{24},(2|3)_{24}))_0=&\Span\{ e_1^*\wedge f_1^*\otimes f_3+e_2^*\wedge f_1^*\otimes f_2+e_2^*\wedge f_2^*\otimes f_1,\\
&e_1^*\wedge e_2^*\otimes e_1+2e_1^*\wedge f_1^*\otimes f_2+e_2^*\wedge f_2^*\otimes f_2,\\
&e_1^*\wedge e_2^*\otimes e_1+e_1^*\wedge f_1^*\otimes f_2-e_1^*\wedge f_2^*\otimes f_3,\\
&e_1^*\wedge f_2^*\otimes f_2,\ e^*_1\wedge f_1^*\otimes f_2-e_2\wedge f_3^*\otimes f_3,\\
&e_1^*\wedge e_2^*\otimes e_2-e_1^*\wedge f_3^*\otimes f_3\}.
\end{align*}
It is not difficult to see that any possible deformation of $(2|3)_{18}$, $(2|3)_{19}$, $(2|3)_{23}$ and $(2|3)_{24}$ gives rise to a non-nilpotent Lie superalgebra. Therefore, they are rigid in the variety $\mathcal{N}_{(2|3)}$. Moreover, the Lie superalgebra $(2|3)_6$ is rigid in the variety $\mathcal{LS}_{(2|3)}$.
\end{proof}

\medskip

\begin{theorem}
The irreducible components of the variety $\mathcal{N}_{(2|3)}$ are:
\begin{enumerate}
\item $\mathcal{C}_1=\overline{O((2|3)_6)}$.
\item $\mathcal{C}_2=\overline{O((2|3)_{18})}$.
\item $\mathcal{C}_3=\overline{O((2|3)_{19})}$.
\item $\mathcal{C}_4=\overline{O((2|3)_{23})}$.
\item $\mathcal{C}_5=\overline{O((2|3)_{24})}$.
\end{enumerate}
\end{theorem}

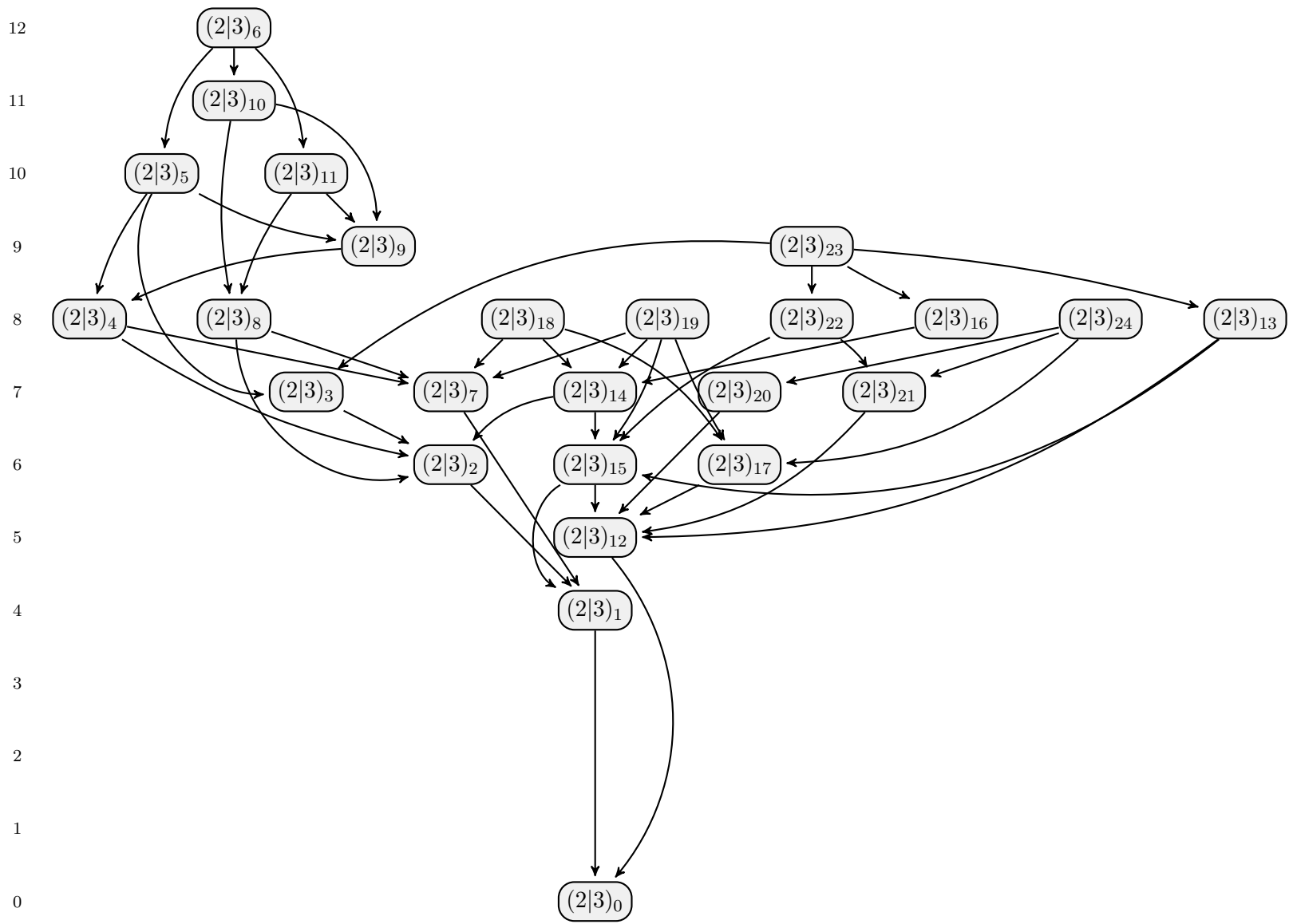
\begin{landscape}

\begin{longtable}{|c|c|}
\caption[]{Non-degenerations Case III}
\label{table:non-deg53}\\
\hline
$\g\not\to\h$ & Reason\\
\hline
\endfirsthead
\caption[]{(continued)}\\
\hline
$\g\not\to\h$ & Reason\\
\hline
\endhead
$(2|3)_{12}\not\to(2|3)_{1};\ (2|3)_{17}\not\to(2|3)_{1}$ & Lemma \ref{lem:invariants} (\ref{inv:gamma})\\ \hline
$(2|3)_{13}\not\to(2|3)_{7};\ (2|3)_{21}\not\to(2|3)_{2},(2|3)_{15},(2|3)_{1};\ (2|3)_{24}\not\to(2|3)_{3},(2|3)_{7},(2|3)_{14},(2|3)_{2},(2|3)_{15},(2|3)_{1};$ & \multirow{2}{*}{Lemma \ref{lem:invariants} (\ref{inv:der}) $i=0$}\\ 
$(2|3)_{20}\not\to(2|3)_{2},(2|3)_{15},(2|3)_{1};\ (2|3)_{16}\not\to(2|3)_{7};\ (2|3)_{22}\not\to(2|3)_{7};\ (2|3)_{23}\not\to(2|3)_{7}$;  & \\ \hline
$(2|3)_{2}\not\to(2|3)_{12};\ (2|3)_{3}\not\to(2|3)_{15},(2|3)_{17},(2|3)_{12};\ (2|3)_{4}\not\to(2|3)_{14},(2|3)_{20},(2|3)_{21},(2|3)_{15},(2|3)_{17},(2|3)_{12};$ & \multirow{8}{*}{Lemma \ref{lem:invariants} (\ref{inv:der}) $i=1$}\\
$(2|3)_{7}\not\to(2|3)_{15},(2|3)_{17},(2|3)_{12};\ (2|3)_{8}\not\to(2|3)_{14},(2|3)_{20},(2|3)_{21},(2|3)_{15},(2|3)_{17},(2|3)_{12};$ & \\ 
$(2|3)_{13}\not\to(2|3)_{20},(2|3)_{21};\ (2|3)_{16}\not\to(2|3)_{20},(2|3)_{21};\ (2|3)_{18}\not\to(2|3)_{20},(2|3)_{21};\ (2|3)_{19}\not\to(2|3)_{20},(2|3)_{21};$ & \\
$(2|3)_{9}\not\to(2|3)_{24},(2|3)_{13},(2|3)_{16},(2|3)_{18},(2|3)_{19},(2|3)_{22},(2|3)_{14},(2|3)_{20},(2|3)_{21},(2|3)_{15},(2|3)_{17},(2|3)_{12};$ & \\ 
$(2|3)_5\not\to(2|3)_{23},(2|3)_{24},(2|3)_{13},(2|3)_{16},(2|3)_{18},(2|3)_{19},(2|3)_{22},(2|3)_{14},(2|3)_{20},(2|3)_{21},(2|3)_{15},(2|3)_{17},(2|3)_{12}$ & \\ 
{$(2|3)_{11}\not\to(2|3)_{23},(2|3)_{13},(2|3)_{16},(2|3)_{18},(2|3)_{19},(2|3)_{22},(2|3)_{24},(2|3)_{14},(2|3)_{20},(2|3)_{21},(2|3)_{15},(2|3)_{17},(2|3)_{12}$} & \\ 
$(2|3)_{10}\not\to(2|3)_{23},(2|3)_{24},(2|3)_{13},(2|3)_{16},(2|3)_{18},(2|3)_{19},(2|3)_{22},(2|3)_{14},(2|3)_{20},(2|3)_{21},(2|3)_{15},(2|3)_{17},(2|3)_{12}$ & \\ 
$(2|3)_{6}\not\to(2|3)_{23},(2|3)_{24},(2|3)_{13},(2|3)_{16},(2|3)_{18},(2|3)_{19},(2|3)_{22},(2|3)_{14},(2|3)_{20},(2|3)_{21},(2|3)_{15},(2|3)_{17},(2|3)_{12}$ & \\ \hline
$(2|3)_{14}\not\to(2|3)_{17};\ (2|3)_{21}\not\to(2|3)_{17};\ (2|3)_{23}\not\to(2|3)_{20}$ & Lemma \ref{lem:invariants} (\ref{inv:centro}) $i=0$\\ \hline
$(2|3)_{20}\not\to(2|3)_{17};\ (2|3)_{13}\not\to(2|3)_{3};\ (2|3)_{16}\not\to(2|3)_{3};\ (2|3)_{22}\not\to(2|3)_{3}$ & Lemma \ref{lem:invariants} (\ref{inv:centro}) $i=1$\\ \hline
$(2|3)_{13}\not\to(2|3)_{14};\ (2|3)_{13}\not\to(2|3)_{2};\ (2|3)_{18}\not\to(2|3)_{3};\ (2|3)_{19}\not\to(2|3)_{2};\ (2|3)_{22}\not\to(2|3)_{14};\ (2|3)_{22}\not\to(2|3)_{2}$ & \multirow{2}{*}{Lemma \ref{lem:invariants} (\ref{inv:ab})}\\ 
$(2|3)_{23}\not\to(2|3)_{4};\ (2|3)_{23}\not\to(2|3)_{8}$ & \\ \hline
$(2|3)_{13}\not\to(2|3)_{17}; (2|3)_{16}\not\to(2|3)_{17};\ (2|3)_{22}\not\to(2|3)_{20};\ (2|3)_{22}\not\to(2|3)_{21};\ (2|3)_{22}\not\to(2|3)_{17};\ (2|3)_{23}\not\to(2|3)_{24}$ & \multirow{2}{*}{Lemma \ref{lem:invariants} (\ref{inv:f})}\\ 
$(2|3)_{23}\not\to(2|3)_{18},(2|3)_{19},(2|3)_{17}$ & \\ \hline
$(2|3)_{7}\not\to(2|3)_{2};\ (2|3)_{4}\not\to(2|3)_{3};\ (2|3)_{8}\not\to(2|3)_{3};\ (2|3)_{19}\not\to(2|3)_{3};\ (2|3)_{9}\not\to(2|3)_{3};\ (2|3)_{10}\not\to(2|3)_5,(2|3)_3;$ & Lemma \ref{lem:invariants} (\ref{inv:trivial})\\ \hline
\multirow{2}{*}{{$(2|3)_{10}\not\to(2|3)_{11}$}} & Lemma \ref{lem:invariants} (\ref{inv:dergen}) $i=0$\\
& $(\alpha,\beta,\gamma)=(0,1,0)$\\ \hline
\multirow{2}{*}{{$(2|3)_{5}\not\to(2|3)_{8};\ (2|3)_{9}\not\to(2|3)_{8}$}} & Lemma \ref{lem:invariants} (\ref{inv:dergen}) $i=0$\\
& $(\alpha,\beta,\gamma)=(0,1,-1)$\\ \hline
\end{longtable}

\scriptsize
\begin{longtable}{|c|lll|}
\caption[]{Degenerations Case III}
\label{table:deg}\\
\hline
$\g\to\h$ & \multicolumn{3}{|c|}{Parametrized Basis}\\
\hline
\endfirsthead
\caption[]{(continued)}\\
\hline
$\g\to\h$ & \multicolumn{3}{|c|}{Parametrized Basis}\\
\hline
\endhead
\multirow{2}{*}{$(2|3)_2\to(2|3)_1$} & $x_1=e_1$, & $x_2=e_2$, & $y_1=f_1$, \\
& $y_2=tf_2$, & $y_3=f_3$. & \\ \hline

\multirow{2}{*}{$(2|3)_{15}\to(2|3)_{12}$} & $x_1=e_1$, & $x_2=e_2$, & $y_1=tf_1$, \\
& $y_2=f_2$, & $y_3=tf_3$. & \\ \hline
\multirow{2}{*}{$(2|3)_{15}\to(2|3)_1$} & $x_1=e_2$, & $x_2=te_1$, & $y_1=f_3$, \\
& $y_2=f_2$, & $y_3=f_1$. & \\ \hline

\multirow{2}{*}{$(2|3)_{17}\to(2|3)_{12}$} & $x_1=e_1$, & $x_2=e_2$, & $y_1=f_1$, \\
& $y_2=tf_2$, & $y_3=f_3$. & \\ \hline

\multirow{2}{*}{$(2|3)_3\to(2|3)_2$} & $x_1=e_1$, & $x_2=e_2$, & $y_1=f_1$, \\
& $y_2=f_2$, & $y_3=tf_3$. & \\ \hline

\multirow{2}{*}{$(2|3)_7\to(2|3)_1$} & $x_1=2e_2$, & $x_2=e_1$, & $y_1=f_2$, \\
& $y_2=tf_1$, & $y_3=f_3$. & \\ \hline

\multirow{2}{*}{$(2|3)_{14}\to(2|3)_2$} & $x_1=e_2$, & $x_2=te_1$, & $y_1=i\left(-\frac{1}{2}f_2+f_3\right)$, \\
& $y_2=\frac{1}{2}f_2+f_3$, & $y_3=f_1$.  &\\ \hline
\multirow{2}{*}{$(2|3)_{14}\to(2|3)_{15}$} & $x_1=e_1$, & $x_2=e_2$, & $y_1=f_1$, \\
 & $y_2=tf_2$, & $y_3=\frac{1}{2}f_2+f_3$. & \\ \hline

\multirow{2}{*}{$(2|3)_{20}\to(2|3)_{12}$} & $x_1=e_1$, & $x_2=te_2$, & $y_1=f_1$, \\
& $y_2=f_2$, & $y_3=f_3$. & \\ \hline

\multirow{2}{*}{$(2|3)_{21}\to(2|3)_{12}$} & $x_1=e_1$, & $x_2=e_2$, & $y_1=f_2$,\\
 & $y_2=t^{-1}f_1$, & $y_3=f_3$. & \\ \hline

\multirow{2}{*}{$(2|3)_4\to(2|3)_7$} & $x_1=\frac{t}{2}\left(e_1+e_2\right)$, & $x_2=\frac{1}{8}\left(-e_1+e_2\right)$, & $y_1=t\left(if_1+f_2\right)$, \\
& $y_2=\frac{1}{2}\left(-if_1+f_2\right)$, & $y_3=f_3$. & \\ \hline
\multirow{2}{*}{$(2|3)_4\to(2|3)_2$} & $x_1=e_1$, & $x_2=t^{-1}\left(e_2-e_1\right)$, & $y_1=f_1$, \\
& $y_2=f_2$, & $y_3=f_3$. & \\ \hline

\multirow{2}{*}{$(2|3)_{8}\to(2|3)_7$} & $x_1=e_1$, & $x_2=e_2$, & $y_1=f_1$, \\
& $y_2=f_2$, & $y_3=tf_3$. & \\ \hline
\multirow{2}{*}{$(2|3)_{8}\to(2|3)_2$} & $x_1=2e_1$, & $x_2=t^{-1}e_2$, & $y_1=i\left(-f_1+f_2\right)$, \\
& $y_2=f_1+f_2$, & $y_3=f_3$. & \\ \hline

\multirow{2}{*}{$(2|3)_{13}\to(2|3)_{15}$} & $x_1=e_1$, & $x_2=e_2$, & $y_1=f_1$,\\
 & $y_2=tf_2$, & $y_3=f_2+f_3$. & \\ \hline

\newpage

\multirow{2}{*}{$(2|3)_{16}\to(2|3)_{14}$} & $x_1=e_1$, & $x_2=e_2$, & $y_1=\frac{t^{-1}}{2}f_1$, \\
& $y_2=t\left(if_2+f_3\right)$, & $y_3=\frac{t^{-1}}{2}\left(-if_2+f_3\right)$. & \\ \hline

\multirow{2}{*}{$(2|3)_{18}\to(2|3)_7$} & $x_1=-e_2$, & $x_2=e_1$, & $y_1=f_3$, \\
& $y_2=f_2$, & $y_3=t^{-1}f_1$. & \\ \hline
\multirow{2}{*}{$(2|3)_{18}\to(2|3)_{14}$} & $x_1=e_1$, & $x_2=-te_2$, & $y_1=f_1$, \\
& $y_2=tf_2$, & $y_3=f_3$. & \\ \hline
\multirow{2}{*}{$(2|3)_{18}\to(2|3)_{17}$} & $x_1=e_1$, & $x_2=e_2$, & $y_1=tf_1$, \\
& $y_2=tf_2$, & $y_3=tf_3$. & \\ \hline

\multirow{2}{*}{$(2|3)_{19}\to(2|3)_7$} & $x_1=-e_1$, & $x_2=e_2$, & $y_1=f_2$, \\
& $y_2=f_3$, & $y_3=t^{-1}f_1$. & \\ \hline
\multirow{2}{*}{$(2|3)_{19}\to(2|3)_{14}$} & $x_1=e_2$, & $x_2=-te_1$, & $y_1=f_1$, \\
& $y_2=tf_3$, & $y_3=f_2$. & \\ \hline
\multirow{2}{*}{$(2|3)_{19}\to(2|3)_{17}$} & $x_1=e_1$, & $x_2=e_2$, & $y_1=tf_1$,\\
 & $y_2=tf_2$, & $y_3=tf_3$.& \\ \hline

\multirow{2}{*}{$(2|3)_{22}\to(2|3)_{21}$} & $x_1=e_1$, & $x_2=t^{-1}e_2$, & $y_1=f_1$, \\
& $y_2=f_2$, & $y_3=f_3$. & \\ \hline
\multirow{2}{*}{$(2|3)_{22}\to(2|3)_{15}$} & $x_1=e_1$, & $x_2=e_2$, & $y_1=f_2$, \\
& $y_2=t^{-1}f_1$, & $y_3=f_3$. & \\ \hline

\multirow{2}{*}{$(2|3)_{24}\to(2|3)_{20}$} & $x_1=te_1$, & $x_2=e_2$, & $y_1=tf_2$, \\
& $y_2=f_1$, & $y_3=f_3$. & \\ \hline
\multirow{2}{*}{$(2|3)_{24}\to(2|3)_{21}$} & $x_1=e_1$, & $x_2=te_2$, & $y_1=f_1$, \\
& $y_2=f_2$, & $y_3=f_3$. & \\ \hline
\multirow{2}{*}{$(2|3)_{24}\to(2|3)_{17}$} & $x_1=e_1$, & $x_2=t^{-1}e_2$, & $y_1=f_1$, \\
& $y_2=tf_3$, & $y_3=f_2$. & \\ \hline


\multirow{2}{*}{$(2|3)_{9}\to(2|3)_4$} & $x_1=e_1$, & $x_2=2e_2$, & $y_1=f_3$, \\
& $y_2=f_2$, & $y_3=tf_1$. & \\ \hline

\multirow{2}{*}{$(2|3)_{23}\to(2|3)_{13}$} & $x_1=e_1$, & $x_2=-2t^{3/2}e_2$, & $y_1=tf_1$, \\
& $y_2=t^{1/4}f_1+t^{5/4}f_3$, & $y_3=tf_2$. & \\ \hline
\multirow{2}{*}{$(2|3)_{23}\to(2|3)_{16}$} & $x_1=e_1$, & $x_2=t^2e_2$, & $y_1=tf_1$, \\
& $y_2=-\frac{1}{2}t^{1/2}f_1+t^{3/2}f_3$, & $y_3=tf_2$. & \\ \hline

\newpage

\multirow{2}{*}{$(2|3)_{23}\to(2|3)_{22}$} & $x_1=e_1$, & $x_2=e_2$, & $y_1=tf_1$,\\
 & $y_2=tf_2$, & $y_3=-\frac{t^{-1}}{2}f_1+tf_3$.& \\ \hline
\multirow{2}{*}{$(2|3)_{23}\to(2|3)_3$} & $x_1=-e_2$, & $x_2=te_1$, & $y_1=\frac{i}{\sqrt{2}}\left(-f_1+f_3\right)$, \\
& $y_2=if_2$, & $y_3=\frac{1}{\sqrt{2}}\left(f_1+f_3\right)$. & \\ \hline

\multirow{2}{*}{$(2|3)_5\to(2|3)_4$} & $x_1=e_1$, & $x_2=e_2$, & $y_1=f_1$, \\
& $y_2=f_2$, & $y_3=tf_3$. & \\ \hline
\multirow{2}{*}{$(2|3)_5\to(2|3)_3$} & $x_1=e_1$, & $x_2=t^{-1}(e_2-e_1)$, & $y_1=f_1$, \\
& $y_2=f_2$, & $y_3=f_3$. & \\ \hline
\multirow{2}{*}{{$(2|3)_{5}\to(2|3)_{9}$}} & $x_1=e_1$, & $x_2=\frac{t^{-1}}{2}e_1+\frac{1}{2}e_2$, & $y_1=t^{1/2}f_1$, \\
& $y_2=t^{-1/2}f_1+f_2$, & $y_3=f_3$. & \\ \hline

\multirow{2}{*}{{$(2|3)_{11}\to(2|3)_{9}$}} & $x_1=e_1$, & $x_2=t^{-2}e_2$, & $y_1=tf_1$, \\
& $y_2=t^{-1}f_2$, & $y_3=f_3$. & \\ \hline
\multirow{2}{*}{{$(2|3)_{11}\to(2|3)_{8}$}} & $x_1=t^{-1}e_1$, & $x_2=e_2$, & $y_1=t^{-1}f_1$, \\
& $y_2=f_2$, & $y_3=f_3$. &\\ \hline

\multirow{2}{*}{$(2|3)_{10}\to(2|3)_{9}$} & $x_1=t^2e_1$, & $x_2=e_2$, & $y_1=t^2f_1$, \\
& $y_2=f_2$, & $y_3=tf_3$. & \\ \hline
\multirow{2}{*}{{$(2|3)_{10}\to(2|3)_{8}$} }& $x_1=-\frac{i\sqrt{2}}{4t(t-1)}e_1$, & $x_2=\frac{1}{2t(t-1)}(e_1+e_2)$, & $y_1=-f_1$, \\
& $y_2=\left(\frac{i\sqrt{2}}{4t(t-1)}\right)(f_1+f_2)-\left(\frac{2t-1}{2t(t-1)}\right)f_3$, & $y_3=\left(\frac{i\sqrt{2}}{4(t-1)}\right)(f_1+f_2)-\left(\frac{1}{2(t-1)}\right)f_3$. & \\ \hline

{$(2|3)_6\to(2|3)_{11}$} & $x_1=t(e_1+e_2)$, & $x_2=\left(\frac{\alpha^2+1}{2}\right)\left(e_1+\frac{1}{\alpha^2}e_2\right)$, & $y_1=t f_3$, \\
$\alpha^2+1=2\alpha\sqrt{t}$ & $y_2=\alpha f_1-\frac{1}{\alpha}f_2+f_3$, & $y_3=\sqrt{t}(f_1-f_2)$. & \\  \hline
\multirow{2}{*}{{$(2|3)_6\to(2|3)_{10}$}} & $x_1=\sqrt{\frac{t}{2}}\left(e_1-\left(\frac{1+\sqrt{2t}}{1-\sqrt{2t}}\right)e_2\right)$, & $x_2=\frac{1}{2}\left(e_1+\left(\frac{1+\sqrt{2t}}{1-\sqrt{2t}}\right)e_2\right)$, & $y_1=\sqrt{\frac{t}{2}}\left(f_1-\left(\sqrt{\frac{1+\sqrt{2t}}{1-\sqrt{2t}}}\right)f_2\right)$, \\
& $y_2=f_1+\left(\sqrt{\frac{1+\sqrt{2t}}{1-\sqrt{2t}}}\right)f_2$, & $y_3=\left(\sqrt{\frac{1+\sqrt{2t}}{2}}\right)f_3$. & \\ \hline

\multirow{2}{*}{$(2|3)_6\to(2|3)_5$} & $x_1=t^2e_1$, & $x_2=e_2$, & $y_1=tf_1$, \\
& $y_2=f_2$, & $y_3=tf_3$. &\\ \hline

\end{longtable}

\normalsize

\begin{center}
\begin{tikzpicture}[->,>=stealth',shorten >=0.08cm,auto,node distance=1.20cm,
                    thick,main node/.style={rectangle,draw,fill=gray!12,rounded corners=1.5ex,font=\sffamily \bf \bfseries },
                    purple node/.style={rectangle,draw, color=purple,fill=gray!12,rounded corners=1.5ex,font=\sffamily \bf \bfseries },
                    orange node/.style={rectangle,draw, color=orange,fill=gray!12,rounded corners=1.5ex,font=\sffamily \bf \bfseries },
                    green node/.style={rectangle,draw, color=green,fill=gray!12,rounded corners=1.5ex,font=\sffamily \bf \bfseries },
                    blue node/.style={rectangle,draw, color=blue,fill=gray!12,rounded corners=1.5ex,font=\sffamily \bf \bfseries },
                    connecting node/.style={circle, draw, color=purple },
                    rigid node/.style={rectangle,draw,fill=black!20,rounded corners=1.5ex,font=\sffamily \tiny \bfseries },style={draw,font=\sffamily \scriptsize \bfseries }]
                    
\node (131)   {$\dim O(\g)$};

\node (121) [below          of=131]      {$12$};
\node (111) [below          of=121]      {$11$};
\node (101) [below          of=111]      {$10$};
\node (91) [below          of=101]      {$9$};
\node (81) [below          of=91]      {$8$};
\node (71) [below          of=81]      {$7$};
\node (61) [below          of=71]      {$6$};
\node (51) [below          of=61]      {$5$};
\node (41) [below          of=51]      {$4$};
\node (31) [below          of=41]      {$3$};
\node (21) [below          of=31]      {$2$};
\node (11) [below          of=21]      {$1$};
\node (01) [below          of=11]      {$0$};

\node (122) [right          of=121]      {};
\node (123) [right          of=122]      {};
\node (124) [right          of=123]      {};
\node (125) [right          of=124]      {};
\node (126) [right          of=125]      {};
\node (127) [right          of=126]      {};
\node (128) [right          of=127]      {};

\node (112) [right          of=111]      {};
\node (113) [right          of=112]      {};
\node (114) [right          of=113]      {};
\node (115) [right          of=114]      {};
\node (116) [right          of=115]      {};
\node (117) [right          of=116]      {};
\node (118) [right          of=117]      {};
\node (119) [right          of=118]      {};
\node (1110) [right          of=119]      {};
\node (1111) [right          of=1110]      {};
\node (1112) [right          of=1111]      {};
\node (1113) [right          of=1112]      {};

\node (102) [right          of=101]      {};
\node (103) [right          of=102]      {};
\node (104) [right          of=103]      {};
\node (105) [right          of=104]      {};
\node (106) [right          of=105]      {};
\node (107) [right          of=106]      {};
\node (108) [right          of=107]      {};
\node (109) [right          of=108]      {};

\node (92) [right          of=91]      {};
\node (93) [right          of=92]      {};
\node (94) [right          of=93]      {};
\node (95) [right          of=94]      {};
\node (96) [right          of=95]      {};
\node (97) [right          of=96]      {};
\node (98) [right          of=97]      {};
\node (99) [right          of=98]      {};
\node (910) [right          of=99]      {};
\node (911) [right          of=910]      {};

\node (82) [right          of=81]      {};
\node (83) [right          of=82]      {};
\node (84) [right          of=83]      {};
\node (85) [right          of=84]      {};
\node (86) [right          of=85]      {};
\node (87) [right          of=86]      {};
\node (88) [right          of=87]      {};
\node (89) [right          of=88]      {};
\node (810) [right          of=89]      {};
\node (811) [right          of=810]      {};
\node (812) [right          of=811]      {};
\node (813) [right          of=812]      {};
\node (814) [right          of=813]      {};
\node (815) [right          of=814]      {};
\node (816) [right          of=815]      {};
\node (817) [right          of=816]      {};

\node (72) [right          of=71]      {};
\node (73) [right          of=72]      {};
\node (74) [right          of=73]      {};
\node (75) [right          of=74]      {};
\node (76) [right          of=75]      {};
\node (77) [right          of=76]      {};
\node (78) [right          of=77]      {};
\node (79) [right          of=78]      {};
\node (710) [right          of=79]      {};
\node (711) [right          of=710]      {};
\node (712) [right          of=711]      {};

\node (62) [right          of=61]      {};
\node (63) [right          of=62]      {};
\node (64) [right          of=63]      {};
\node (65) [right          of=64]      {};
\node (66) [right          of=65]      {};
\node (67) [right          of=66]      {};
\node (68) [right          of=67]      {};
\node (69) [right          of=68]      {};
\node (610) [right          of=69]      {};

\node (52) [right          of=51]      {};
\node (53) [right          of=52]      {};
\node (54) [right          of=53]      {};
\node (55) [right          of=54]      {};
\node (56) [right          of=55]      {};
\node (57) [right          of=56]      {};
\node (58) [right          of=57]      {};

\node (42) [right          of=41]      {};
\node (43) [right          of=42]      {};
\node (44) [right          of=43]      {};
\node (45) [right          of=44]      {};
\node (46) [right          of=45]      {};
\node (47) [right          of=46]      {};
\node (48) [right          of=47]      {};

\node (02) [right          of=01]      {};
\node (03) [right          of=02]      {};
\node (04) [right          of=03]      {};
\node (05) [right          of=04]      {};
\node (06) [right          of=05]      {};
\node (07) [right          of=06]      {};
\node (08) [right          of=07]      {};

\node [main node] (5a)  [right of =123]                      {$(2|3)_6$ };

\node [main node] (13a)  [right of =113]                      {$(2|3)_{10}$ };

\node [main node] (50)  [right of =102]                      {$(2|3)_5$ };
\node [main node] (13x)  [right of =104]                      {$(2|3)_{11}$ };

\node [main node] (130)  [right of =95]                      {$(2|3)_{9}$ };
\node [main node] (25)  [right of =911]                      {$(2|3)_{23}$ };

\node [main node] (4)  [right of =81]                      {$(2|3)_4$ };
\node [main node] (11)  [right of =83]                      {$(2|3)_8$ };
\node [main node] (15)  [right of =817]                      {$(2|3)_{13}$ };
\node [main node] (18)  [right of =813]                      {$(2|3)_{16}$ };
\node [main node] (20)  [right of =87]                      {$(2|3)_{18}$ };
\node [main node] (21)  [right of =89]                      {$(2|3)_{19}$ };
\node [main node] (24)  [right of =811]                      {$(2|3)_{22}$ };
\node [main node] (27)  [right of =815]                      {$(2|3)_{24}$ };

\node [main node] (3)  [right of =74]                      {$(2|3)_3$ };
\node [main node] (6)  [right of =76]                      {$(2|3)_7$ };
\node [main node] (16)  [right of =78]                      {$(2|3)_{14}$ };
\node [main node] (22)  [right of =710]                      {$(2|3)_{20}$ };
\node [main node] (23)  [right of =712]                      {$(2|3)_{21}$ };

\node [main node] (2)  [right of =66]                      {$(2|3)_{2}$ };
\node [main node] (17)  [right of =68]                      {$(2|3)_{15}$ };
\node [main node] (19)  [right of =610]                      {$(2|3)_{17}$ };

\node [main node] (14)  [right of =58]                      {$(2|3)_{12}$ };

\node [main node] (1)  [right of =48]                      {$(2|3)_{1}$ };

\node [main node] (0)  [right of =08]                      {$(2|3)_{0}$ };

\path[every node/.style={font=\sffamily\small}]

(5a)  edge [bend right=0, color=black] node{}  (13a)
(5a)  edge [bend right=-20, color=black] node{}  (13x)
(5a)  edge [bend right=20, color=black] node{}  (50)

(13a)  edge [bend right=-40, color=black] node{}  (130)
(13a)  edge [bend right=10, color=black] node{}  (11)

(50)  edge [bend right=10, color=black] node{}  (130)
(50)  edge [bend right=10, color=black] node{}  (4)
(50)  edge [bend right=60, color=black] node{}  (3)

(13x)  edge [bend right=10, color=black] node{}  (11)
(13x)  edge [bend right=0, color=black] node{}  (130)

(130)  edge [bend right=10, color=black] node{}  (4)
(25)  edge [bend right=20, color=black] node{}  (3)
(25)  edge [bend right=-5, color=black] node{}  (15)
(25)  edge [bend right=3, color=black] node{}  (18)
(25)  edge [bend right=0, color=black] node{}  (24)

(4)  edge [bend right=0, color=black] node{}  (6)
(4)  edge [bend right=10, color=black] node{}  (2)

(11)  edge [bend right=50, color=black] node{}  (2)
(11)  edge [bend right=0, color=black] node{}  (6)

(15)  edge [bend right=-18, color=black] node{}  (14)
(15)  edge [bend right=-25, color=black] node{}  (17)

(18)  edge [bend right=0, color=black] node{}  (16)

(20)  edge [bend right=0, color=black] node{}  (6)
(20)  edge [bend right=0, color=black] node{}  (16)
(20)  edge [bend right=-20, color=black] node{}  (19)

(21)  edge [bend right=0, color=black] node{}  (6)
(21)  edge [bend right=0, color=black] node{}  (16)
(21)  edge [bend right=-10, color=black] node{}  (17)
(21)  edge [bend right=5, color=black] node{}  (19)

(24)  edge [bend right=-10, color=black] node{}  (23)
(24)  edge [bend right=10, color=black] node{}  (17)

(27)  edge [bend right=0, color=black] node{}  (22)
(27)  edge [bend right=0, color=black] node{}  (23)
(27)  edge [bend right=-20, color=black] node{}  (19)


(3)  edge [bend right=0, color=black] node{}  (2)

(6)  edge [bend right=0, color=black] node{}  (1)

(16)  edge [bend right=20, color=black] node{}  (2)
(16)  edge [bend right=0, color=black] node{}  (17)

(22)  edge [bend right=0, color=black] node{}  (14)

(23)  edge [bend right=-20, color=black] node{}  (14)

(2)  edge [bend right=0, color=black] node{}  (1)

(17)  edge [bend right=0, color=black] node{}  (14)
(17)  edge [bend right=60, color=black] node{}  (1)

(19)  edge [bend right=0, color=black] node{}  (14)

(14)  edge [bend right=-40, color=black] node{}  (0)

(1)  edge [bend right=0, color=black] node{}  (0);

\end{tikzpicture}
\end{center}
\end{landscape}

\normalsize

\section{Rigid Nilpotent Lie Superalgebras}

As we mentioned previously, the study of rigid elements within varieties is very interesting since their orbit closures provide irreducible components. In \cite{GKN}, the authors proved that the Lie superalgebra $K^{2,m}\in \mathcal{N}_{(2|m)}$, for odd $m$, defined by
$$\[e_1,f_i\]=f_{i+1},\qquad\[f_j,f_{m+1-j}\]=(-1)^{j+1}e_2$$
for $1\leq i\leq m-1$ and $1\leq j\leq \frac{m+1}{2}$ has an open orbit in $\mathcal{N}_{(2|m)}$ and is therefore  rigid in $\mathcal{N}_{(2|m)}$. In the case of Lie algebras, the Vergne Conjecture states that there are no rigid nilpotent Lie algebras of dimension $n$ in the variety of Lie algebras of dimension $n$. This conjecture have been proved for Lie algebras of $\rk\geq 1$ (see \cite{HT}) and for very few cases in $\rk=0$ (see for instance \cite{TV}). In the setting of Lie superalgebras this conjecture is false: In \cite{AH}, we prove that there exists a rigid nilpotent Lie superalgebra of dimension $(2|2)$ by showing that $(H^2((2|2)_1)_0=0$. By Propositions \ref{prop:rigid2}, \ref{prop:rigid3}, \ref{prop:rigid4}, \ref{prop:rigid51}, \ref{prop:rigid52} and \ref{prop:rig_23}, we can see that in every dimension $\leq 5$ there are rigid nilpotent Lie superalgebras. This make us wonder whether this is the general case for every dimension. 

\medskip

In particular, we can prove:

\begin{lemma}
There exist rigid nilpotent Lie superalgebras of dimension $(1|n)$ for all $n\in \N$.
\end{lemma}
\begin{proof}
Let $\mathcal{B}=\{e_1,f_1,\dots,f_n\}$ be a homogeneous basis for $V=V_0\oplus V_1$ of dimension $(1|n)$. We define the following bracket:
$$\[f_{i},f_{i}\]=e_1,\ \text{for }i=1,\dots,n.$$
Then clearly $\g=(V,\[\cdot,\cdot\])$ is a 2-step nilpotent Lie superalgebra. In fact, this is a Heisenberg Lie superalgebra with even center.

\medskip

We will compute the group $(H^2(\g,\g))_0$:

Consider possible even 2-cochains of $\g$:
\begin{enumerate}[(a)]
\item $\varphi=e_1^*\wedge f_j^*\otimes f_l$. Then $$d^2(\varphi)=\displaystyle\sum_{i=1}^{n}f_{i}^*\wedge f_{i}^*\wedge f_j^*\otimes f_l+e_1^*\wedge f_j^*\wedge f_l^*\otimes e_1\neq 0.$$ Since elements of the sum cannot be eliminated, $\varphi$ cannot be part of a 2-cocycle.
\item $\varphi=f_i^*\wedge f_j^*\otimes e_1$. Then $d^2(\varphi)=0$.
\end{enumerate}

We thus obtain that $(Z^2(\g,\g))_0=\Span\{ f_i^*\wedge f_j^*\otimes e_1 \}$, but $d^1(f_i^*\otimes f_j)=f_i^*\wedge f_j^*\otimes e_1$, which implies that $(H^2(\g,\g))_0=0$. Therefore $\g$ is a nilpotent rigid Lie superalgebra.
\end{proof}

\bigskip

{\bf Acknowledgements: }
M.A.Alvarez was supported by MINEDUC-UA project, code ANT 1755 and ``Fondo Puente de Investigaci\'on de Excelencia'' FPI-18-02 from University of Antofagasta.
I. Hern\'andez was supported by grants FOMIX-CONACYT  YUC-2013-C14-221183 and 222870.

\end{document}